\documentclass[11pt]{amsart}
\usepackage{amssymb}
\usepackage{amsmath}
\usepackage{amsthm}
\usepackage{mathtools}
\usepackage[textwidth=16cm, textheight=24cm, marginratio=1:1]{geometry}

\usepackage{newtxtext} 
\usepackage[english]{babel}
\usepackage[utf8]{inputenc}
\usepackage[T1]{fontenc}
\usepackage{todonotes}
\usepackage{xcolor}

\usepackage[backref=page]{hyperref}
\hypersetup{
    colorlinks,
    linkcolor={red!80!black},
    citecolor={blue!80!black},
    urlcolor={blue!80!black}
}
\usepackage{setspace}
\usetikzlibrary{matrix,arrows, cd, calc}  
\newsavebox{\pullback}
\sbox\pullback{%
\begin{tikzpicture}%
\draw (0,0) -- (1ex,0ex);%
\draw (1ex,0ex) -- (1ex,1ex);%
\end{tikzpicture}}
\usepackage{tkz-fct} 
\usepackage{tikz-cd}
\usetikzlibrary{cd}
\usetikzlibrary{arrows}
\usepackage{quiver}
\usepackage{enumitem}
\usepackage{booktabs}
\usepackage{xfrac}
\usepackage{dsfont}
\numberwithin{equation}{section}
\newtheorem{theorem}[equation]{Theorem}
\newtheorem{corollary}[equation]{Corollary}
\newtheorem{lemma}[equation]{Lemma}

\newtheorem{proposition}[equation]{Proposition}
\theoremstyle{definition}
\newtheorem{definition}[equation]{Definition}

\newtheorem{condition}[equation]{Condition}

\DeclareMathOperator{\Spec}{Spec}
\DeclareMathOperator{\Sspec}{\mathcal{S}pec}

\DeclareMathOperator{\Supp}{Supp}

\DeclareMathOperator{\Hom}{Hom}

\DeclareMathOperator{\Sch}{Sch}
\DeclareMathOperator{\Set}{Set}
\DeclareMathOperator{\Nat}{Nat}
\DeclareMathOperator{\Funct}{Funct}

\newcommand{\kk}{\Bbbk}%

\DeclareMathOperator{\Hilb}{Hilb}
\DeclareMathOperator{\Quot}{Quot}

\DeclareMathOperator{\Bilin}{Bilin}
\DeclareMathOperator{\GL}{GL}%
\DeclareMathOperator{\Sym}{Sym}%
\DeclareMathOperator{\PP}{\mathbb P}

\DeclareMathOperator{\Gr}{Gr}%
\newcommand{\mm}{\mathfrak{m}}%

\DeclareMathOperator{\pr}{pr}

\newcommand{\tM}{\widetilde{M}}
\newcommand{\tF}{\widetilde{F}}
\newcommand{\tK}{\widetilde{K}}
\newcommand{\tQ}{\widetilde{Q}}
\newcommand{\keps}{\Bbbk[\varepsilon]}
\newcommand{\Seps}{S[\varepsilon]}

\newcommand{\Z}{\mathbb Z}

\newcommand{\A}{\mathbb A}
\newcommand{\CC}{\mathbb C}

\newcommand{\U}{\mathcal{U}}
\newcommand{\I}{\mathcal{I}}
\newcommand{\Oo}{\mathcal{O}}
\newcommand{\Q}{\mathcal Q}
\newcommand{\B}{\mathcal B}
\newcommand{\T}{\mathcal{T}}
\newcommand{\M}{\mathcal{M}}
\newcommand{\Nn}{\mathcal{N}}
\newcommand{\E}{\mathcal{E}}
\newcommand{\W}{\mathcal{W}}

\newcommand{\Zz}{\mathcal{Z}}
\newcommand{\nn}{\mathfrak{n}}

\newcommand{\xrightarrowdbl}[2][]{%
  \xrightarrow[#1]{#2}\mathrel{\mkern-14mu}\rightarrow
}

\begin{document}

\title{The geometry and singularities of the Bilinear scheme}
\author{Weronika Obcowska}
\email{w.obcowska@student.uw.edu.pl}
\thanks{Supported by NCN grant number 2023/50/E/ST1/00336.}
\date{\today{}}
\maketitle

\begin{abstract}
    The goal is to study the geometry of the Bilinear scheme $\Bilin_{d_1,d_2,d_3}^{r_1, r_2}(\A^n)$ 
    introduced by Joachim Jelisiejew. 
    This functor can be viewed as a generalization of the Quot scheme, 
    giving the moduli space of bilinear maps of locally free modules. 
    We describe the relation to the Quot scheme by proving that the Bilinear functor can be realized 
    as a closed subfunctor of a product of Quot schemes, 
    hence the Bilinear functor is representable by a closed subscheme of the product of Quot schemes.  
    We use this result to compute the tangent space to the Bilinear scheme representing $\Bilin_{d_1,d_2,d_3}^{r_1, r_2}(\A^n)$. 
    We define two types of loci: the locus corresponding to tuples of points, 
    and the totally degenerate locus. 
    The first locus gives the main irreducible component of the Bilinear scheme. 
    We use the theory of minimal border rank tensors and secant varieties, and find that 
    $\Bilin_{d_1,d_2,d_3}^{r_1, r_2}(\A^n)$ is reducible for all $n$ whenever $r_i \geq d \geq 3$. 
    We describe the $\kk$-points of $\Bilin_{2,2,2}^{2, 2}(\A^1)$ in detail. 
\end{abstract}

\section{Introduction}
We are interested in the moduli space that arises in a natural way in the classification problem 
of concise minimal border rank tensors in $\CC ^m \otimes \CC ^m \otimes \CC ^m$. 
As explained in \cite{jlp}, this question remains open and turns out to be extremely difficult. 
The contributions of \cite{jlp, jjjj} solve the classification problem in the special case $m=5$. 
One of the tools introduced as a new invariant of a concise tensor is the 111-algebra structure. 
A tensor $\tau \in V_1 \otimes V_2 \otimes V_3$ with a 111-algebra $A$ 
induces a map $\tau \colon V_1 ^{\vee} \otimes V_2 ^{\vee} \rightarrow V_3$ that is $A$-bilinear in a canonical way. 
We want to study a moduli space parametrizing such bilinear maps, 
which motivates the construction of the \emph{Bilinear functor}. 
Following \cite[Problem XXXVIII]{jj}, we define the Bilinear functor 
\begin{equation*}
    \Bilin_{d_1,d_2,d_3}^{r_1, r_2} (\A^n) \colon \Sch_{/\Bbbk}^{op} \rightarrow \Set
\end{equation*}
on $\kk$-points as
\[
  \Bilin_{d_1,d_2,d_3}^{r_1, r_2}(\A^n)(\Bbbk) = 
  \left\{\begin{minipage}[c]{\widthof{\text{quotients of $S$-modules $S^{\oplus r_i} \xrightarrowdbl[]{p_i} M_i = \sfrac{S^{\oplus r_i}}{K_i}$ for $i=1,2$}}}
    \centering
    \text{\text{quotients of $S$-modules $S^{\oplus r_i} \xrightarrowdbl[]{p_i} M_i = \sfrac{S^{\oplus r_i}}{K_i}$ for $i=1,2$ }}\\
    \text{together with a surjection $M_1 \otimes_S M_2 \xrightarrowdbl[]{\pi} M_3 = \sfrac{S^{\oplus r_1 r_2}}{K_3}$}\\
    \text{such that $\dim_{\Bbbk} M_j =d_j$ for $j=1,2,3$}
  \end{minipage}\right\}
\]
\noindent
where $S = \kk[x_1, \ldots, x_n]$, over an arbitrary infinite algebraically closed field $\kk$. 
This construction gives a generalization of the Hilbert and Quot schemes of points.  
We discuss the connection to the Quot scheme in more detail. 
To this end, we define a closed embedding into a product of Quot schemes 
\begin{align*}
    \Bilin_{d_1,d_2,d_3}^{r_1, r_2}(\A^n) \hookrightarrow \Quot_{d_1}^{r_1}(\A^n) \times \Quot_{d_2}^{r_2}(\A^n) \times \Quot_{d_3}^{r_1 r_2}(\A^n)
\end{align*}
\noindent
and use it to show that the Bilinear functor is representable, which means that now we can refer to it as the Bilinear scheme. 
We review the definition of tangent space to the Quot scheme and use this result to
compute the tangent space to the Bilinear scheme. 

For the Hilbert and Quot schemes, we describe the irreducible components corresponding to tuples of points. 
They are called the smoothable component of the Hilbert scheme and the principal component of the Quot scheme. 
Those components have been studied extensively due to their connection to other areas of mathematics. 
The smoothable component of the Hilbert scheme has applications in combinatorics \cite{jb, haiman}, 
and the principal component of the Quot scheme arises naturally in the study of the variety of commuting matrices \cite{jjks}. 
We generalize those constructions to the setting of $\Bilin_{d,d,d}^{r_1, r_2}(\A^n)$ 
and define its main irreducible component as follows. 
For $p_1 \colon S^{\oplus r_1} \twoheadrightarrow M_1$, assume that $M_1$ corresponds to a tuple of points. 
Then there must be isomorphisms $M_2 \simeq M_1$ and $M_3 \simeq M_1$, 
so all three $S$-modules correspond to the same tuple of points. 
Moreover, we find that the third map $\pi \colon M_1 \otimes_S M_2 \twoheadrightarrow M_3$ must be a uniquely defined isomorphism. 
In order to compute the dimension of the main component, 
we study certain open subschemes of the principal component of the Quot scheme. 
We show that a subscheme of this form is isomorphic 
to a vector bundle over the smoothable component of the Hilbert scheme. 
Using this result, we prove that the main component of the Bilinear scheme is of dimension $nd + (r_1 -1)d + (r_2 -1)d$.  

We claim that the Bilinear scheme is reducible in general and prove it by describing the locus that must lie outside of the main component. 
For $\Quot_{d}^{r}(\A^n)$, define the totally degenerate locus as corresponding to modules of the form $(\sfrac{S}{\mm})^{\oplus d} $, 
for a fixed maximal ideal $\mm \lhd S$. 
We show that it gives a closed subscheme $\Zz_{\mm}$ 
and generalize this to the totally degenerate locus of $\Bilin_{d,d,d}^{r_1, r_2}(\A^n)$ defined as follows. 
Assume that we have surjections $p_i \colon S^{\oplus r_i} \twoheadrightarrow (\sfrac{S}{\mm})^{\oplus d}$ for $i=1,2$, 
then the third map is of the form $\pi \colon (\sfrac{S}{\mm})^{\oplus d^2} \twoheadrightarrow (\sfrac{S}{\mm})^{\oplus d}$. 
We prove that the dimension of this locus is $ (r_1-d)d + (r_2 -d)d + (d^2 - d)d$. 

In general, the main component and the totally degenerate locus have different dimensions. In particular, 
the totally degenerate locus cannot be contained in the main component whenever 
\begin{equation*}
    nd + (r_1 -1)d + (r_2 -1)d < (r_1-d)d + (r_2 -d)d + (d^2 - d)d,
\end{equation*}
which translates to the dimension of the totally degenerate locus being greater than that of the main component. 
This gives examples of parameters $r_i,d,n$ such that 
$\Bilin_{d,d,d}^{r_1, r_2}(\A^n)$ is reducible, which is the case for $n < d^2 - 3d +2$. 

In order to determine if $\Bilin_{d,d,d}^{r_1, r_2}(\A^n)$ is irreducible in the remaining cases, 
we relate it to the secant variety of the Segre embedding. 
Let $[p_1,p_2,\pi] \in \Bilin_{d,d,d}^{r_1, r_2}(\A^n)(\kk)$, 
where $\pi \colon M_1 \otimes_S M_2 \twoheadrightarrow M_3$. 
This map defines an $M_3$-concise tensor $\mu_{\pi} \in M_1^{\vee} \otimes M_2^{\vee} \otimes M_3$. 
Tensors coming from points of the main irreducible component of $\Bilin_{d,d,d}^{r_1, r_2}(\A^n)$ 
are of border rank $d$, which means that they arise as 
limits of multiplication tensors of algebras corresponding to tuples of $d$ points. 
Secant varieties give a geometric interpretation of border rank: 
the affine cone over the $d$-th secant variety $\sigma_d (\PP ^{m-1} \times \PP ^{m-1} \times \PP ^{m-1})$ 
parametrizes tensors from $\kk^m \times \kk^m \times \kk^m$ that are of border rank at most $d$. 
It follows that the main irreducible component of $\Bilin_{d,d,d}^{r_1, r_2}(\A^n)$ 
gives tensors corresponding to points of the $d$-th secant variety 
$\sigma_d (\PP ^{d-1} \times \PP ^{d-1} \times \PP ^{d-1}) \subseteq \PP (\kk^d \otimes \kk^d \otimes \kk^d)$. 
Irreducibility of $\Bilin_{d,d,d}^{r_1, r_2}(\A^n)$ is equivalent to 
$\sigma_d (\PP ^{d-1} \times \PP ^{d-1} \times \PP ^{d-1})$ being the whole ambient space, 
which is the case only for $d \leq 2$. 
This allows us to find all parameters such that $\Bilin_{d,d,d}^{r_1, r_2}(\A^n)$ is irreducible. 
In particular, $\Bilin_{d,d,d}^{r_1, r_2}(\A^n)$ is reducible for all $n$ if $r_i \geq d \geq 3$. 

We give a detailed description of the special case of two points $\Bilin_{2,2,2}^{2, 2}(\A^1) $, 
our scheme is irreducible. 
Let $S = \kk[x]$. 
A point $[p_1,p_2, \pi] \in \Bilin_{2,2,2}^{2, 2}(\A^1)(\kk)$ is given by 
$p_i \colon S^{\oplus 2} \twoheadrightarrow M_i$, for $i=1,2$, 
together with $\pi \colon M_1 \otimes M_2 \twoheadrightarrow M_3$
such that $\dim_{\kk} M_j = 2$ for $j=1,2,3$. 
If an $S$-module $M$ with $\dim_{\kk} M = 2$ is cyclic then it must satisfy 
$M \simeq \sfrac{S}{((x-c_1)(x-c_2))}$ or $M \simeq \sfrac{S}{(x-c)^2}$
for $c,c_1,c_2 \in \kk$ such that $c_1 \neq c_2$. 
Otherwise, if $M$ is not cylic, then $M \simeq (\sfrac{S}{(x-c)})^{\oplus 2}$ for $c \in \kk$. 
It suffices to consider cases $\sfrac{S}{(x(x-1))} $, $\sfrac{S}{(x)^2} $ 
and $(\sfrac{S}{(x)})^{\oplus 2}$. 
We start by discussing the cases of cyclic modules. 
Assume $M_1 \simeq \sfrac{S}{(x(x-1))}$, so that we get a point from the main component. 
Then also $M_2 \simeq M_3 \simeq \sfrac{S}{(x(x-1))}$ and the map 
$\pi \colon M_1 \otimes M_2 \twoheadrightarrow M_3$ gives a concise tensor of rank 2. 
This tensor is isomorphic to the multiplication tensor of $\sfrac{S}{(x(x-1))} $.
Now for $i=1,2$ assume $M_i \simeq \sfrac{S}{(x)^2}$. We find that $M_3 \simeq \sfrac{S}{(x)^2}$ and the map 
$\pi \colon M_1 \otimes M_2 \twoheadrightarrow M_3$ gives a concise tensor of rank 3, 
the multiplication tensor of $\sfrac{S}{(x)^2} $. 
By expressing it as a limit of tensors from the previous case we show that its border rank is 2. 
Next we consider modules that are not necessarily cyclic. 
Let $M_1 \simeq \sfrac{S}{I}$. 
We only need to discuss $p_1 \colon S^{\oplus 2} \twoheadrightarrow M_1 \simeq \sfrac{S}{(x)^2}$. 
Then either $M_2 \simeq \sfrac{S}{(x)^2}$ or $M_2 \simeq (\sfrac{S}{(x)})^{\oplus 2}$. 
Suppose that the latter holds. Then the map 
$\pi \colon M_1 \otimes M_2 \twoheadrightarrow M_3$ gives a non-concise tensor of rank 2, 
the multiplication tensor of the $S$-module $(\sfrac{S}{(x)})^{\oplus 2}$. 
This tensor fails to be concise on the first coordinate. 
Likewise, for $p_1 \colon S^{\oplus 2} \twoheadrightarrow M_1 \simeq (\sfrac{S}{(x)})^{\oplus 2}$ 
and $p_2 \colon S^{\oplus 2} \twoheadrightarrow M_2 \simeq \sfrac{S}{(x)^2}$, 
the third map gives a rank 2 tensor that fails to be concise on the second coordinate. 
Finally, we consider the last case, points from the totally degenerate locus. 
Let $[p_1, p_2, \pi]$ be given by $p_i \colon S^{\oplus 2} \twoheadrightarrow M_i \simeq (\sfrac{S}{(x)})^{\oplus 2}$, for $i=1,2$. 
Then $\pi \colon M_1 \otimes M_2 \simeq (\sfrac{S}{(x)})^{\oplus 4} \twoheadrightarrow M_3 \simeq (\sfrac{S}{(x)})^{\oplus 2}$ 
is defined by a full rank $4 \times 2$ matrix. 
No new types of tensors can arise this way and, 
by irreducibility,  
all tensors that arise can be expressed as limits of concise tensors of rank 2.

\section*{Acknowledgements}

I would like to thank my advisor, Joachim Jelisiejew, 
for the support, inspiration and guidance I received in the process of writing my thesis. 
I am also grateful to my friend, Jakub Jagiełła, 
for numerous discussions that played a crucial role in developing my intuition for tensors.

\section{Preliminaries}

We will use the following notation and adhere to the following conventions. 
We are working over an infinite algebraically closed field $\Bbbk = \Bar{\Bbbk}$. 
For a $\Bbbk$-scheme $X$, write $\A^n_X = \A^n_{\Bbbk} \times _{\Bbbk} X$. 
When the scheme $X$ is affine, say $X= \Spec A$, we will write $\A^n_A = \A^n_{\Spec A}$. 
We will write $\A^n$ for $\A_{\Bbbk}^n = \Spec S$, 
where $S = \Bbbk [ x_1, \dots, x_n ]$. 
For a ring $A$, we will write $S_A = S \otimes_{\Bbbk} A$. 
For an ideal $I$ in the ring $A$, we will write $I \lhd A$.
For an $A$-module $M$, we will write $M^{\sim}$ for the quasicoherent sheaf determined by $M$ on $\Spec A$. 
For the category of $\Bbbk$-schemes, we will write $\Sch_{/ \Bbbk}$ 
and for the category of sets, we will write $\Set$. 
Let $X, \, Y$ be $\Bbbk$-schemes. We will write $X \times Y = X \times_{\Bbbk} Y$ 
for the product over $\Spec \Bbbk$ 
and $\Hom(X, \, Y) = \Hom_{\Sch_{/ \Bbbk}}(X, \, Y)$ for the morphisms over $\Spec \Bbbk$. 
Let $\mathcal{C}, \, \mathcal{D}$ be categories.  
We will write $\Funct(\mathcal{C}^{op}, \, \mathcal{D})$ for the category 
of contravariant functors between $\mathcal{C}$ and $\mathcal{D}$. 
Let $F, \, G \colon \mathcal{C} \rightarrow \mathcal{D}$ be functors.  
We will write $\Nat(F, \, G)$ for the natural transformations between $F$ and $G$.

\subsection{Subfunctors}\label{ref:secsubfunct}

Every scheme $X$ determines a contravariant functor 
$\underline{X} \colon \Sch_{/ \Bbbk}^{op} \rightarrow \Set$ 
defined as the set of morphisms $\underline{X}(Y) = \Hom (Y, \, X)$. 
Assigning the functor of points to the scheme determines a covariant functor 
$\Sch_{/ \Bbbk} \rightarrow \Funct( \Sch_{/ \Bbbk}^{op}, \, \Set)$. 

\begin{theorem}[Yoneda Lemma]\label{r:thmyoneda}
    Let $F \colon \Sch_{/ \Bbbk}^{op} \rightarrow \Set$ be a contravariant functor. 
    Let $X$ be a $\Bbbk$-scheme. 
    There is a bijection $\Nat(\underline{X}, \, F) \simeq F(X)$ natural in $X$. 
\end{theorem}
\noindent
This allows us to identify schemes with their functors of points. 
Moreover, we can view the category of $\Bbbk$-schemes as a full subcategory 
of the category of functors $\Funct( \Sch_{/ \Bbbk}^{op}, \, \Set)$. 

\begin{corollary}[Yoneda embedding]\label{r:thmyonedaemb}
    The covariant functor 
    $\Sch_{/ \Bbbk} \rightarrow \Funct( \Sch_{/ \Bbbk}^{op}, \, \Set)$ 
    given by $X \, \mapsto \, \underline{X}$ 
    is fully-faithful. 
\end{corollary}

\begin{definition}\label{r:defreprfunct}
    We say that a functor $F \colon \Sch_{/ \Bbbk}^{op} \rightarrow \Set$ 
    is \emph{representable} if there is a scheme $X$ together with 
    a natural isomorphism $\psi \colon \underline{X} \xrightarrow{\simeq} F$. 
    In this case we say that $F$ is representable by $X$. 
\end{definition}
\noindent
By Corollary \ref{r:thmyonedaemb}, the representing scheme $X$ is unique up to isomorphism. 
Assume that the isomorphism $\psi \colon \underline{X} \xrightarrow{\simeq} F$ exists, 
then it follows from Theorem \ref{r:thmyoneda} that there is a unique element $\xi \in F(X)$, 
called the \emph{universal element}, 
such that for every $f \colon Y \rightarrow X$, there is a commutative diagram 
\[\begin{tikzcd}
	{\underline{X}(X)} && {F(X)} && {\mathds{1}_{X}} & {\xi = \psi_{X}(\mathds{1}_X)} \\
	{\underline{X}(Y)} && {F(Y)} && {(Y \xrightarrow{f} X)} & {\psi_Y(f) = Ff(\xi)}
	\arrow["{\psi_{X}}", from=1-1, to=1-3]
	\arrow["{f^*}"', from=1-1, to=2-1]
	\arrow["Ff", from=1-3, to=2-3]
	\arrow[maps to, from=1-5, to=1-6]
	\arrow[maps to, from=1-5, to=2-5]
	\arrow[maps to, from=1-6, to=2-6]
	\arrow["{\psi_Y}", from=2-1, to=2-3]
	\arrow[maps to, from=2-5, to=2-6]
\end{tikzcd}\]
\noindent
It means that for every $Y$ and every element $c \in F(Y)$, 
there is a unique map $f \colon Y \rightarrow X$ such that 
$c$ can be obtained as the pullback of the universal element along $f$. 

Now we generalize the notions of open and closed subschemes to the setting of functors. 
As general references we use \cite{eh, str}.

\begin{definition}\label{r:defsubfunct}
    Let $G$ and $F$ be functors $Sch^{op} \rightarrow Set$. Then $G$ is a \emph{subfunctor} of $F$ if:
    \begin{enumerate}
        \item there is inclusion of sets $G(T) \subseteq F(T)$ for every scheme $T$,

        \item for every morphism of schemes $t \colon T' \rightarrow T$, 
        the induced map $G(t) \colon G(T) \rightarrow G(T')$ is the restriction of 
        $F(t) \colon F(T) \rightarrow F(T')$. 
    \end{enumerate}
    \noindent
    In other words, the inclusion $G \hookrightarrow F$ is required to be a natural transformation of functors.
\end{definition}
\noindent
Assume moreover that $Y \subseteq X$ is an open (resp. closed) subscheme. 
Let $f \colon T \rightarrow X$ be a morphism of schemes. 
Then the preimage $f^{-1}(Y)$ defines an open (resp. closed) subscheme 
$T \times_X Y$ of $T$, as in the following commutative diagram: 
\[\begin{tikzcd}
	{T \times_X Y} && Y \\
	\\
	T && X
	\arrow["{f'}", from=1-1, to=1-3]
	\arrow["\begin{array}{c} \text{open} \\ \text{(resp. closed)} \end{array}"', from=1-1, to=3-1]
	\arrow["{\iota'}"{description}, shift left=3, draw=none, from=1-1, to=3-1]
	\arrow["\begin{array}{c} \text{open} \\ \text{(resp. closed)} \end{array}", from=1-3, to=3-3]
	\arrow["\iota"{description}, shift right=3, draw=none, from=1-3, to=3-3]
	\arrow["f", from=3-1, to=3-3]
\end{tikzcd}\]
\noindent
Consider the corresponding functors of points. 
For a scheme $W$ we have the following fiber product of sets:
\[\begin{tikzcd}
	{\underline{T}(W) \times_{\underline{X}(W)} \underline{Y}(W)} && {\underline{Y}(W)} \\
	\\
	{\underline{T}(W)} && {\underline{X}(W)}
	\arrow[from=1-1, to=1-3]
	\arrow[from=1-1, to=3-1]
	\arrow["{\iota_W}", from=1-3, to=3-3]
	\arrow["{f_W}", from=3-1, to=3-3]
\end{tikzcd}\]
\noindent
This can be viewed as the special case of the fiber product of functors. 
\begin{definition}[{\cite[Definition VI-4]{eh}}]
    Let $F, \, G, \, H$ be functors $\Sch^{op} \rightarrow \Set$. 
    Let $g \colon G \rightarrow F$ and $h \colon H \rightarrow F$ be natural transformations. 
    Then the \emph{fiber product of functors} $G, \, H$ over $F$ is a functor 
    $G \times_F H \colon \Sch^{op} \rightarrow \Set$ defined as 
    \begin{align*}
        \begin{split}
            (G \times_F H)(W) &  =  G(W) \times_{F(W)} H(W) \\
            & =  \left\{ (a,b) \in G(W) \times H(W) \mid g_W(a) = h_W(b) \text{ in } F(W) \right\}.
        \end{split}
    \end{align*}
\end{definition}
\noindent
In the special case of functors of points, the fiber product 
$\underline{T} \times_{\underline{X}} \underline{Y}$ is representable 
by the fiber product of schemes $T \times_X Y$. 
To see this, consider morphisms $g \colon W \rightarrow T$, $h \colon W \rightarrow Y$ 
such that $g \circ f = h \circ \iota$. 
By the universal property of pullback diagram, 
there is a unique morphism $\varphi_W \colon W \rightarrow T \times_X Y$ 
such that $h = \varphi \circ \iota'$ and $g = \varphi \circ \iota '$. 
This can be written in terms of functors of points in the following way: 
\begin{align*}
    \begin{split}
        \underline{T}(W) \times_{\underline{X}(W)} \underline{Y}(W) &  = \Hom(W, \, T) \times_{\Hom(W, \, X)} \Hom(W, \, Y) \\
        & \simeq \Hom(W, \, T \times_X Y) \, = \, \underline{(T \times_X Y)}(W).
    \end{split}
\end{align*}
The isomorphisms $\varphi_W$ are natural in $W$, so they give components of a natural transformation 
\begin{equation*}
    \varphi \colon \underline{T} \times_{\underline{X}} \underline{Y} \xrightarrow{\simeq} \underline{T \times_X Y}.
\end{equation*}
This can be generalized to arbitrary functors $F, G  \colon \Sch^{op} \rightarrow \Set$, 
which gives rise to the notion of open and closed subfunctors.
\begin{definition}\label{r:defopclsubf}
    Let $G$ be a subfunctor $F$. We say that it is an \emph{open} (resp. closed) \emph{subfunctor} 
    if for every scheme $T$ and every object $\xi \in G(T)$, 
    the induced fiber product of functors $\underline{T} \times_G F$ 
    is representable by an open (resp. closed) subscheme of $T$.
\end{definition}

\begin{proposition}\label{r:defopclsubfdef2}
    Let $G$ be a subfunctor of $F$. 
    Then $G$ is an open (resp. closed) subfunctor if the following holds. 
    For any scheme $T$ and $\xi \in F(T)$, there is 
    an open (resp. closed) subscheme $U \subseteq T$ such that 
    for any $f \colon T' \rightarrow T$ we have:
    \begin{equation*}
        \text{the element } f^{*} \xi \in F(T') \text{ is in } G(T') \
        \Longleftrightarrow \
        f \text{ factors through } U \subseteq T.
    \end{equation*}

\end{proposition}
\noindent
As a corollary we have the following result that we will use without explicit reference. 

\begin{proposition}
    Let $G$ be an open (resp. closed) subfunctor of $F$ such that 
    $F$ is representable by the scheme $X$. 
    Then $G$ is representable by an open (resp. closed) subscheme of $X$. 
\end{proposition}
\noindent
In the setting of Proposition \ref{r:defopclsubfdef2}, 
the fiber product from Definition \ref{r:defopclsubf} is representable by $U$. 
By \cite[Proposition VI-2]{eh}, it suffices to verify the conditions 
of Proposition \ref{r:defopclsubfdef2} for affine schemes. 
In the case of closed subfunctors we can rephrase it in the following way. 

\begin{proposition}\label{r:prpclsubfcriterion}
    Let $G$ be a subfunctor of $F$. 
    Then $G$ is a closed subfunctor if the following holds. 
    For any affine scheme $\Spec A$ and $\xi \in F(\Spec A)$ there is 
    an ideal $I \lhd A$ such that 
    for any morphism of affine schemes $f \colon \Spec B \rightarrow \Spec A$ we have:
    \begin{equation*}
        \text{the element } f^{*} \xi \in F(\Spec B) \text{ is in } G(\Spec B)
        \ \Longleftrightarrow \
        f \text{ factors through } \Spec(\sfrac{A}{I}) \subseteq \Spec A.
    \end{equation*}
\end{proposition}
\noindent
We will now describe two useful constructions that give 
a closed and an open subfunctor, respectively. 
We start with the \emph{functor of zeros}.
\begin{proposition}[The functor of zeros]\label{r:prpfunctofzeros}
    Let $Y$ be a scheme. Consider a morphism of locally free $\Oo_Y$-modules 
    $\varphi \colon \E_1 \rightarrow \E_2$. 
    Define $\Zz(\varphi) \colon \Sch^{op} \rightarrow \Set$, the functor of zeros of $\varphi$, as 
    \begin{equation*}
        \Zz(\varphi)(X) = \{ f \colon X \rightarrow Y \, \mid \, f^* \varphi \colon f^* \E_1 \rightarrow f^* \E_2 \text{ is the zero map} \}.
    \end{equation*}
    The functor $\Zz(\varphi)$ is representable by a closed subscheme of $Y$. 
\end{proposition}

\begin{proof}
    We start by showing that the assertion holds in the special case of free sheaves. 
    Let $U \subseteq Y$ be an open subscheme such that 
    \begin{equation*}
        (\E_1)_{\vert_U} = \Oo_U^{\oplus N_1} \text{\quad and \quad} (\E_2)_{\vert_U} = \Oo_U^{\oplus N_2}.
    \end{equation*}
    \noindent
    Then $\varphi_{\vert_U} = (\varphi_{ij})_{i,j}$ is given by an $N_2 \times N_1$ matrix 
    and the pullback gives 
    \begin{equation*}
        f^* \varphi = (f^{\#}(\varphi_{ij}))_{i,j} \colon f^*(\E_1)_{\vert_{U}} \simeq \Oo_X^{\oplus N_1} \, \rightarrow \, f^* (\E_2)_{\vert_U} \simeq \Oo_X^{\oplus N_2},
    \end{equation*}
    \noindent
    where $f^{\#} \colon \Oo_U \rightarrow f_{*} \Oo_X$ is the induced map. 
    Then $f^* \varphi$ is the zero map if and only if the morphism $f$ factors 
    through the closed subscheme of $Y$ defined by the vanishing of 
    sections $\varphi_{ij}$, denoted $Z(\varphi_{ij})$. 
    Equivalently, the functor $\Zz(\varphi_{\vert_U})$ is representable by 
    the scheme $Z(\varphi_{ij})$. 

    Now we prove the claim in full generality by constructing a covering of $\Zz(\varphi)$ by open representable subfunctors. 
    Consider an open cover $\{ U_i \}$ of $Y$. 
    Let $\Zz_i = \Zz(\varphi) \times_Y U_i$ be the fiber product of functors. 
    For every scheme $X$ there is a commutative diagram 
    \[\begin{tikzcd}
	{X \times_Y U_i} & {X \times_{\Zz(\varphi)} \Zz_i} && {\Zz_i = \Zz(\varphi) \times_Y U_i} && {U_i} \\
	& X && {\Zz(\varphi)} && Y
	\arrow["\simeq", from=1-1, to=1-2]
	\arrow[from=1-2, to=1-4]
	\arrow[from=1-2, to=2-2]
	\arrow[from=1-4, to=1-6]
	\arrow[from=1-4, to=2-4]
	\arrow[from=1-6, to=2-6]
	\arrow[from=2-2, to=2-4]
	\arrow[from=2-4, to=2-6]
    \end{tikzcd}\]
    hence $\Zz_i$ is an open subfunctor of $\Zz(\varphi)$. 
    Let $\kk$ be a field, we will show that $\bigcup_i \Zz_i(\kk) = \Zz(\varphi)(\kk)$. 
    For each $i$ we have a diagram 
    \[\begin{tikzcd}
	{\Zz_i(\kk)} && {U_i(\kk)} \\
	{\Zz(\varphi)(\kk)} && {Y(\kk)}
	\arrow["{\pi_i}", from=1-1, to=1-3]
	\arrow["\rho"', from=1-1, to=2-1]
	\arrow["{\alpha_i}", from=1-3, to=2-3]
	\arrow["\beta", from=2-1, to=2-3]
    \end{tikzcd}\] 
    \noindent
    where 
    \begin{equation*}
        \Zz_i(\kk) = \left\{ (f, p) \in \Zz(\varphi)(\kk) \times U_i(\kk) \ \vert \ \alpha_i(f) = \beta(p) \right\}
    \end{equation*} 
    for $f \colon \Spec \kk \rightarrow U_i$ and $p \colon \Spec \kk \rightarrow X$ such that $p^* \varphi = 0$. 
    The map $\alpha_i$ is defined as the composition 
    $\alpha_i(f) = (\Spec \kk \xrightarrow{f} U_i \xhookrightarrow{\alpha_i} X)$. 
    Then
    \begin{equation*}
        \Zz(\varphi)(\kk) = \{ p \colon \Spec \kk \rightarrow X \, \mid \, p^* \varphi \colon (\E_1)_{\vert p} \rightarrow (\E_2)_{\vert p} \text{ is zero} \},
    \end{equation*}
    so $\beta$ is the natural inclusion. 
    The set $\Zz_i(\kk)$ can be viewed as those $\kk$-points $p \colon \Spec \kk \rightarrow X$ 
    that give zero maps on fibers $(\E_1)_{\vert p} \rightarrow (\E_2)_{\vert p}$ 
    and that are also in $U_i$. 
    Since $\{ U_i \}$ is an open cover of $X$, 
    every $p \in X(\kk)$ is in some $U_i(\kk)$, 
    hence $\bigcup_i \Zz_i(\kk) = \Zz(\varphi)(\kk)$. 
    It follows that the functors $\Zz_i$ form an open cover of $\Zz(\varphi)$. 
    From the special case of free sheaves it follows 
    that for a cover $\{ U_i \}$ of $X$ that trivializes the sheaves $\E_1, \ \E_2$, 
    the functors $\Zz_i$ are representable. 
    The functor $\Zz(\varphi)$ satisfies Zariski descent, 
    so by \cite[Theorem VI-14]{eh} the closed subschemes $Z_i \subseteq Y$ 
    glue to give a closed subscheme representing $\Zz(\varphi)$. 
\end{proof}
\noindent
We can describe the sheaf of ideals associated to the closed subscheme 
that represents the functor $\Zz(\varphi)$ from Proposition \ref{r:prpfunctofzeros}. 
We start with the special case of a morphism of free sheaves over an affine scheme. 
Let $Y = \Spec A$ and consider the map of free $A$-modules 
$\varphi \colon M = A^{\oplus m} \rightarrow N = A^{\oplus n}$ 
corresponding to the morphism of free sheaves $\varphi^{\sim} \colon M^{\sim} \rightarrow N^{\sim}$. 
As in the proof of Proposition \ref{r:prpfunctofzeros}, 
this map is given by an $n \times m$ matrix $(\varphi_{ij})$. 
Let $I(\varphi) $ denote the ideal generated by the entries of the matrix $(\varphi_{ij})$. 
Then 
\begin{align}\label{r:eqfitting}
    \begin{split}
        \Zz(\varphi)(X) & = \{ f \colon X \rightarrow \Spec A \, \mid \, f^* \varphi = 0 \} \\
        & \simeq \{ f^{\#} \colon A \rightarrow \Gamma(X, \Oo_X) \, \mid \, f^{\#}(\varphi_{ij}) = 0 \text{ for all } i,j \} \\
        & \simeq  \{ f^{\#} \colon A \rightarrow \Gamma(X, \Oo_X) \, \mid \, f^{\#} \text{ factors through } A \twoheadrightarrow \sfrac{A}{I(\varphi)} \} \\ 
        & \simeq \Hom( \sfrac{A}{I(\varphi)}, \, \Gamma(X, \Oo_X)) \\
        & \simeq \Hom(X, \, \Spec(\sfrac{A}{I(\varphi)})).
    \end{split}
\end{align}
\noindent
In the case of an arbitrary scheme $Y$ and 
a morphism of locally free $\Oo_X$-modules $\varphi \colon \M \rightarrow \Nn$, 
the isomorphisms in (\ref{r:eqfitting}) give an explicit description of 
the closed subscheme representing the functor $\Zz(\varphi)$. 
The construction is natural in $A$, so the ideals $I(\varphi)$ glue to 
give a quasicoherent sheaf of ideals $\I(\varphi)$ on $Y$ 
which we call the \emph{Fitting ideal} of $\varphi$. 
This notation is slightly simplified with respect to the literature: 
the full name is the 0-th Fitting ideal of $\varphi$. 
\begin{definition}\label{r:deffittingideal}
    Let $Y$ be a scheme. Let $\varphi : \mathcal{M} \xrightarrow{} \mathcal{N}$ be 
    a morphism of locally free sheaves of rank $m,n$, respectively.
    We define the \emph{Fitting ideal} on $Y$ as the quasicoherent sheaf of ideals $\I(\varphi)$  
    associated to the closed subscheme representing the functor $\Zz(\varphi)$. 
\end{definition}

We will now use Hilbert's Nullstellensatz to construct examples of open subfunctors. 
Let the functor $F \colon \Sch_{/\Bbbk}^{op} \rightarrow \Set$ 
be representable by a scheme $X$. 
If we require the scheme $X$ to have some additional properties, 
then to define an open subfunctor $G$ of $F$, it suffices to specify its $\kk$-points $G(\kk)$. 

\begin{lemma}\label{r:lemnullst}
    Let $X$ be a $\kk$-scheme locally of finite type. 
    Then for every open subset $U_{\kk}$ there is an open subscheme $U \subseteq X$ 
    such that $U(\kk) = U_{\kk}$. 
\end{lemma}

\begin{proof}
    For every point $p_i$ of $U_{\kk}$ there is an open affine neighborhood 
    $U_i = \Spec (A_i) \subseteq X$ such that $A_i$ is a finitely generated $\kk$-algebra. 
    By Hilbert's Nullstellensatz there is a bijection between 
    open subsets of $U_i(\kk)$ and open subschemes of $U_i$. 
    The intersection $ U_i(\kk) \cap U_{\kk}$ is open in $U_i(\kk)$, 
    so there is an open subscheme $U_i' \subseteq U_i$ such that 
    $U_i'(\kk) = U_i(\kk) \cap U_{\kk}$. 
    The subschemes $\{ U_i' \}_i$ glue to give an open subscheme $U$ 
    such that $U_{\kk} = U(\kk) = \bigcup _i U_i'(\kk)$. 
\end{proof}

\begin{proposition}\label{r:prpkpointsext}
    Let $X$ be a $\kk$-scheme locally of finite type. 
    Take an open subset $Y_{\kk} \subseteq X(\kk)$. 
    Then there is a representable functor 
    $G \colon \Sch_{/\Bbbk}^{op} \rightarrow \Set$ such that 
    $G(\kk) = Y_{\kk}$. 
\end{proposition}

\begin{proof}
    Let $Y$ be the open subscheme such that $Y(\kk)= Y_{\kk}$, by Lemma \ref{r:lemnullst}.
    Let $G$ be the functor of points $\underline{Y}$, 
    then $Y(\kk)=Y_{\kk} = G(\kk)$. 
    This defines an open subfunctor of $F$. 
\end{proof}
\noindent
We will say that the open subset $Y_{\kk}$ extends to the functor $G$. 

\subsection{Tangent space}\label{r:sectangent}

We give the definition of the tangent space to a scheme that can be generalized to functors. 
Let $D = \keps = \sfrac{\Bbbk[t]}{t^2}$ be the ring of dual numbers and write 
$S[\varepsilon] = S \otimes_{\Bbbk} D$. 
The following definition comes from deformation theory. 

\begin{proposition}[{\cite[Chapter VI.1.3]{eh}}]\label{r:deftangent}
    Let $X$ be a $\Bbbk$-scheme, let $p \colon \Spec \Bbbk \rightarrow X$ be a $\Bbbk$-point. 
    Then the Zariski tangent space to $X$ at $p$ 
    consists of morphisms $\widetilde{p} \colon \Spec D \rightarrow X$ such that 
    the following diagram commutes: 

    \[\begin{tikzcd}
	{\Spec D} && X \\
	& {\Spec \Bbbk}
	\arrow["{\widetilde{p}}", from=1-1, to=1-3]
	\arrow["\alpha", from=2-2, to=1-1]
	\arrow["p"', from=2-2, to=1-3]
    \end{tikzcd}\]
\end{proposition}
\noindent
From the perspective of functors of points, 
the tangent space to $\underline{X}$ at $p \in \underline{X}(\Bbbk)$ 
consists of $\widetilde{p} \in \underline{X}(D)$ such that 
the pullback $\alpha^{*} \colon \underline{X}(D) \rightarrow \underline{X}(\Bbbk)$ 
sends $\widetilde{p}$ to $p$. 

\subsection{Irreducibility}

We recall the notion of dimension at a point, which we will use while discussing 
irreducibility of $\Bilin_{d_1,d_2,d_3}^{r_1, r_2}(\A^n)$. 

\begin{definition}
    Let $X$ be a topological space. 
    Define the dimension of $X$ at a point $p \in X$ as 
    \begin{equation}
        \dim_p X = \min \, \{ \dim U \mid U \subseteq X \text{ is an open neighborhood of } p \},
    \end{equation}
    \noindent
    where the minimum is taken over all open neighborhoods of $p$. 
\end{definition}

\begin{lemma}\label{r:lemirreddim}
    Let $X$ be a topological space. 
    If $X$ is irreducible then $\dim_p X$ is the same at every $p \in X$. 
\end{lemma}

\begin{proof}
    Assume that $X$ is irreducible. 
    Let $V_1, V_2$ be any two non-empty open subsets of $X$. 
    Then $V_1 \cap V_2 \neq \varnothing$ by assumption. 
    Take a point $p \in V_1 \cap V_2$, it must be $\dim_p V_1 = \dim_p V_2$. 
    Suppose there is a point $q \in V_1$ such that $\dim_q V_1 \geq \dim_p V_1$. 
    By assumption, every open neighborhood of $q$ intersects every open neighborhood of $p$. 
    Take a neighborhood $U_1 \in p$ with $\dim U_1 = \dim _p V_1$, 
    then $U_1$ is also an open neighborhood of $q$, so 
    $\dim_q V_1 \leq \dim U_1 = \dim_p V_1$, hence $\dim_q V_1 = \dim_p V_1$. 
    This reasoning shows that $\dim _p X = \dim _q X$ for all $p,q \in X$, 
    since every open neighborhood of $p$ is also an open neighborhood of $q$. 
\end{proof}
\noindent
Suppose we claim that a topological space $X$ is reducible. 
In order to prove this, we can try to find an open subset $U \subseteq X$ 
and a closed subset $Z \subseteq X$ such that $\dim U < \dim Z$. 
By Lemma \ref{r:lemirreddim}, this shows that $\overline{U} \cap Z = \varnothing$. 
This will be our strategy for the proof of reducibility of $\Bilin_{d_1,d_2,d_3}^{r_1, r_2}(\A^n)$. 

\subsection{Hilbert scheme}\label{r:sechilb}

Consider a functor $\Hilb_d (\A^n) \colon \Sch_{/\Bbbk}^{op} \rightarrow \Set$ 
defined on $\Bbbk$-points as 
\begin{equation*}
    \Hilb_d(\A^n)(\Bbbk) = \{ I \lhd S \mid \dim _{\Bbbk} \sfrac{S}{I} = d \}.
\end{equation*}
\noindent
For an arbitrary $\Bbbk$-scheme $X$ the Hilbert functor is defined as 
    \begin{equation*}
        \Hilb_d(\A^n)(X) = \{ \text{closed subschemes } Z \subseteq X \times_{\Bbbk} \A^n \mid \text{ $\Oo_Z$ is a locally free $\Oo_X$-module of rank $d$} \}. 
    \end{equation*}
\noindent
This gives rise to the Hilbert functor of $d$ points in the affine $n$-space. 
It is representable by a scheme $Hilb$, called the Hilbert scheme. 
Let $\psi \colon \underline{Hilb} \rightarrow \Hilb_d(\A^n)$ 
be the natural isomorphism and 
let $[\eta]$ denote the universal element, the image of 
$\psi_{Hilb}(\mathds{1}_{Hilb})$ in $\Hilb_d(\A^n)(Hilb)$. 
We will write 
\begin{equation}\label{r:equnivhilb}
    [\eta] = [\Zz \xhookrightarrow{\eta} Hilb \times \A^n].
\end{equation}
\noindent
The universal element $[\eta]$ corresponds to the closed subscheme 
$\mathcal{Z} \xhookrightarrow{\eta} Hilb \times \A^n$ such that $\Oo_{\Zz}$ is locally free of rank $d$ over $Hilb$ 
and such that for every scheme $X$, every element of $\Hilb_d(\A^n)(X)$ 
is obtained as the pullback of $\mathcal{Z}$ by some unique 
$f \colon X \rightarrow Hilb$. 
This is encoded by the following commutative diagram:
\[\begin{tikzcd}
	{\mathcal{Z} \times_{\kk} X} && {X \times \A_{\Bbbk}^n} \\
	{\mathcal{Z}} && {Hilb \times \A_{\Bbbk}^n}
	\arrow[hook, from=1-1, to=1-3]
	\arrow[from=1-1, to=2-1]
	\arrow["f \times \mathds{1}_{\A^n}", from=1-3, to=2-3]
	\arrow["{\eta}", hook, from=2-1, to=2-3]
\end{tikzcd}\]
\noindent
If $X = \Spec A$ for a $\Bbbk$-algebra $A$, 
we write $\Hilb_d(\A^n)(A)$ for $\Hilb_d(\A^n)(\Spec A)$. 
In that case we have 
\begin{align*}
    \begin{split}
        \Hilb_d(\A^n)(A) & = \{ Z \subseteq \Spec A \times \A^n \, \mid \, \text{ $\Oo_Z$ is a locally free $\Oo_{\Spec A}$-module of rank $d$} \} \\
        & \simeq \{ I \lhd S_A \, \mid \, \sfrac{S_A}{I} \text{ is locally free $A$-module of rank $d$}\}.
    \end{split}
\end{align*}

We say that a closed subscheme $Z \subseteq \A^n$ corresponds to a tuple of $d$ points 
if it consists of $d$ distinct $\Bbbk$-points of $\A^n$. 
If $Z = \Spec(\sfrac{S}{I})$ is a tuple of points, then 
\begin{equation*}
    \Oo_Z \simeq \sfrac{S}{I} \simeq \sfrac{S}{\I( \{ p_1, \dots, p_d \} )} \simeq \Bbbk^d,
\end{equation*}
\noindent
where $p_i \in \A^n(\Bbbk)$ are distinct points. 
We denote the set of ideals $I \lhd S$ corresponding to tuples of points by $T_{\Bbbk}$ 
and think of it as a subset of $\Hilb_d(\A^n) (\Bbbk)$. 
This subset gives rise to an open subfunctor of the Hilbert scheme 
representable by an open subscheme $T$ of $Hilb$. 
The dimension of $T$ is $nd$. 

We can view $T_{\kk}$ as consisting of those 
$\Bbbk$-points $p \colon \Spec \Bbbk \rightarrow Hilb$ 
that give the isomorphism of $\Bbbk$-algebras
\begin{equation*}
(\alpha _{*} \Oo_Z)_{\vert p} \simeq \Bbbk ^d, \quad
\text{where } \alpha \colon Z \xhookrightarrow{\eta} Hilb \times \A^n \xrightarrow{pr_1} Hilb.
\end{equation*}
\noindent
The scheme $Hilb$ satisfies the assumptions of Proposition \ref{r:prpkpointsext} 
and $T_{\kk}$ is open in $\Hilb_d(\A^n)(\Bbbk)$.
It follows that $T_{\kk}$ induces a functor 
$\T \colon \Sch^{op}_{/ \Bbbk} \rightarrow \Set$ such that 
$\T(\kk)=T_{\kk}$ on $\Bbbk$-points. 
It is defined as the open subfunctor 
$\T \subseteq \Hilb_d(\A^n)$, $\T = \underline{T}$ for an open subscheme $T \subseteq Hilb$ such that $T(\kk) = T_{\kk}$. 
We define the \emph{smoothable component} of the Hilbert scheme 
as the closure of $T$ in $Hilb$. 
The pullback $[\eta'] = \iota^* [\eta]$ by the inclusion $\iota \colon \T \hookrightarrow Hilb$ 
gives the universal element of $\T$. 
It corresponds to a closed subscheme 
\begin{equation}\label{r:equnivt}
    \Zz_{T} \xhookrightarrow{\eta'} \A^n_{T}
\end{equation}
\noindent
such that the composition 
$\beta \colon \Zz_{T} \xhookrightarrow{\eta'} \A^n_{T} \xrightarrow{\pr_1} T$ 
makes $\Oo_{\Zz_T}$ locally free of rank $d$ over $\Oo_T$.

\subsection{Quot scheme}\label{r:secquot}

Consider a functor $\Quot_d ^r (\A^n) \colon \Sch_{/\Bbbk}^{op} \rightarrow \Set$ 
defined on $\Bbbk$-points as
\begin{align*}
    \Quot_d ^r(\A^n)(\Bbbk) = \{ \text{quotients of $S$-modules } S^{\oplus r} \xrightarrowdbl[]{\pi} M \mid \dim_{\Bbbk} M = d \}_{/ \sim}
\end{align*}
\noindent
where we identify surjections $\pi \sim \pi'$ if $\ker \pi = \ker \pi'$. 
Given an arbitrary $\Bbbk$-scheme $X$ we have
\[
  \Quot_{d}^{r}(\A^n)(X) = 
  \left\{\begin{minipage}[c]{\widthof{\text{ equivalence classes of quasicoherent quotients }}}
    \centering
    \text{\text{equivalence classes of quasicoherent quotients}}\\
    \text{$\Oo_X^{\oplus r} \otimes_{\Bbbk} S \xrightarrowdbl[]{\pi} \M$ of $(\Oo_X \otimes_{\Bbbk} S)$-modules}\\
    \text{such that $\M$ is locally free of rank $d$ over $\Oo_X$}
  \end{minipage}\right\}
\]
\noindent
This functor is representable by the scheme $Quot$, 
called the Quot scheme of points in the affine space. 
From the natural isomorphism $\psi \colon \underline{Quot} \rightarrow \Quot_d ^r(\A^n)$, 
we get the universal element 
\begin{align}\label{r:equnivquot}
    \begin{split}
        [\xi] & = \psi_{Quot}(\mathds{1}_{Quot}) \\
        & = \big[ \Oo_{Quot}^{\oplus r} \otimes_{\Bbbk} S \xrightarrowdbl[]{\xi} \M \big] \in \Quot_d ^r(\A^n)(Quot).
    \end{split}
\end{align}
\noindent
For every scheme $X$, every element of $\Quot_d ^r(\A^n)(X)$ 
can be obtained as the pullback of $[\xi]$ 
by some unique $f \colon X \rightarrow Quot$, 
which we define as
\begin{align*}
    f^* [\xi] = \big[ (f^* \Oo_{Quot}^{\oplus r}) \otimes_{\Bbbk} S \xrightarrowdbl[]{f^* \xi} f^* \M \big] \in \Quot_d ^r(\A^n)(X).
\end{align*}
\noindent
We write $\Quot_d ^r(\A^n)(A)$ for $\Quot_d ^r(\A^n) (\Spec A)$. 
In that case 
\begin{equation*}
    \Quot_d ^r(\A^n)(A) = \bigl\{ A^{\oplus r} \otimes S \simeq S_A^{\oplus r} \twoheadrightarrow M \mid M \text{ is locally free of rank $d$ as an $A$-module} \bigr\}_{/ \sim}.
\end{equation*}

Now we recall the characterization of the tangent space to the Quot scheme. 
Fix the following notation: 
\begin{align*}
    \begin{split}
        & [p] = \big[ F = S^{\oplus r} \xrightarrowdbl[]{p} M = \sfrac{F}{K} \big] \in \Quot_d ^r (\A^n)(\Bbbk), \\
        & [\widetilde{p}] = \big[ \tF = S[\varepsilon]^{\oplus r} \xrightarrowdbl[]{\widetilde{p}} \tM = \sfrac{\tF}{\tK} \big] \in \Quot_d ^r (\A^n)(\keps).
    \end{split}
\end{align*}
As in Proposition \ref{r:deftangent}, $\widetilde{p}$ determines a tangent vector at $p$ if 
$\widetilde{p}$ restricts to $p$ on $\Spec \Bbbk$:
\begin{align*}
    \begin{split}
        \alpha^* [\widetilde{p}] & =[p] \\
        & = \big[ \tF \otimes _{\keps} \Bbbk \xrightarrowdbl[]{\widetilde{p} \otimes 1} \tM \otimes_{\keps} \Bbbk \, \simeq \, \sfrac{\tF \otimes \Bbbk}{\tK \otimes \Bbbk} \big] 
    \end{split}
\end{align*}
\noindent
or in other words, if we have the isomorphisms $\sfrac{\tM}{\varepsilon \tM} \simeq M$ and $\sfrac{\tK}{\varepsilon \tK} \simeq K$. 
By definition, the points of $\Quot_d ^r(\A^n)(\keps)$ correspond to equivalence classes of surjections 
$[\tF \twoheadrightarrow \tM]$ such that $\tM$ is locally free of rank $d$ over $\keps$. 
The module $\tM$ can be equivalently characterized as free or as flat. 

\begin{theorem}\label{r:thmquottangent}
    The tangent space to $\Quot_d ^r(\A^n)$ at $[p]$ is isomorphic to $\Hom_S (K, \, \sfrac{F}{K})$.
\end{theorem}

\noindent
For a $\kk$-point $[p \colon F \rightarrow M = \sfrac{F}{K}]$ and a map $\phi \colon K \rightarrow M$, 
the corresponding tangent vector is defined as
\begin{equation*}
    \big[ \widetilde p \colon \widetilde F \rightarrow \widetilde M = \sfrac{\widetilde F}{\widetilde K} \big],
\end{equation*}
where 
\begin{equation}\label{r:eqtK}
    \tK = \{ k+ \varepsilon g \ \vert \ k \in K, \ g \in F \text{ such that } g+K = \varphi(k) \}.
\end{equation}

We can generalize the smoothable component of the Hilbert scheme 
to the setting of the Quot scheme. 
For an ideal $I \lhd S$, consider surjections 
$S^{\oplus r} \twoheadrightarrow M$ such that 
$M$ comes from an algebra $\sfrac{S}{I}$ corresponding to a tuple of points,
which means that there is an isomorphism $M \simeq \sfrac{S^{\oplus 1}}{I}$. 
This gives rise to an open subscheme $U$ of $Quot$ 
such that the dimension of $U$ is $nd + (r-1)d$. 
The closure of this subscheme is called the \emph{principal component}. 
Consider a subset 
\begin{align*}
    U_{\kk} & = \{ S^{\oplus r} \twoheadrightarrow M \mid M \text{ corresponds to a tuple of $d$ points} \} _{/ \sim} \\
    & = \{ S^{\oplus r} \twoheadrightarrow M \mid M \simeq \sfrac{S^{\oplus 1}}{I} \text{ for an ideal } I \lhd S \} _{/ \sim} \subseteq \Quot_d ^r(\A^n)(\kk).
\end{align*}
We have the universal element of $\Quot_d ^r(\A^n)$ that we denote 
$[\xi] = [ \Oo_{Quot}^{\oplus r} \otimes S \xrightarrowdbl{\xi} \M ]$ as in (\ref{r:equnivquot}). 
We can view the set $U_{\kk}$ as consisting of $\Bbbk$-points 
$p \colon \Spec \Bbbk \rightarrow Quot$ such that 
\begin{equation*}
    p^{*} [\xi] = \big[ (\Oo_{Quot}^{\oplus r} \otimes S)_{\vert p} \simeq S^{\oplus r} \ \xrightarrowdbl{\bar{\xi}} \ \M_{\vert p} \simeq \sfrac{S}{I} \big]
\end{equation*}
when we take the pullback. 
The set $U_{\kk}$ is open in $\Quot_d ^r(\A^n)(\kk)$, so it 
gives rise to an open subfunctor $\U$ of $\Quot_d ^r(\A^n)$ 
such that $\U(\kk) = U_{\kk} = U(\kk)$, 
where $U$ is an open subscheme of $Quot$ corresponding to $U_{\kk}$. 
We sketch the construction of $\U$ and compute its dimension. 

We start by covering the set $U_{\kk}$, $v \in \Bbbk^{\oplus r}$, by 
open sets $\U_v(\Bbbk)$ defined in the following way. 
The vector $v \in \Bbbk^{\oplus r}$ comes from a global section 
$s_v \colon \Oo_{Quot} \otimes S \rightarrow \Oo_{Quot}^{\oplus r} \otimes S$ 
such that the pullback along a $\Bbbk$-point 
$p \colon \Spec \Bbbk \rightarrow Quot$ gives a map
\begin{equation*}
    (\Oo_{Quot} \otimes S)_{\vert p} \simeq S \rightarrow  (\Oo_{Quot}^{\oplus r} \otimes S)_{\vert p}  \simeq S^{\oplus r}, \quad 1 \mapsto v \in \Bbbk ^{\oplus r} \subseteq S^{\oplus r}. 
\end{equation*}
For $v \in \Bbbk^{\oplus r}$ of this form we define a subset $\U_v(\Bbbk) \subseteq U_{\kk}$ as 
\begin{equation*}
    \U_v(\Bbbk) = \{ S^{\oplus r} \xrightarrowdbl[]{\pi} M \mid \pi(Sv) = M \text{ and $M$ corresponds to a tuple of $d$ points} \}_{/ \sim} \subseteq U_{\kk}.
\end{equation*}
\begin{proposition}
    $U_{\kk}$ is covered by $\{ \U_v(\Bbbk) \}_{v \in \Bbbk^{\oplus r}}$.
\end{proposition}
\noindent
The sets of the form $\U_v(\Bbbk)$ are open in $\Quot_d ^r(\A^n)(\Bbbk)$ 
and since $U_{\kk} = \bigcup_v \U_v(\Bbbk)$, 
it follows that $U_{\kk}$ is open in $\Quot_d ^r(\A^n)(\Bbbk)$. 
The proof can be broken down into two steps:
\begin{enumerate}
    \item Fix $v \in \Bbbk^{\oplus r}$ and define 
    \begin{equation}\label{r:equprime}
        \U_v'(\Bbbk) = \{ S^{\oplus r} \xrightarrowdbl[]{\pi} M \ \vert \ \pi(Sv) = M \}_{/ \sim} \ \subseteq \ \Quot_d ^r(\A^n)(\Bbbk).
    \end{equation}
    \noindent
    Then $\U_v(\Bbbk) = \U_v'(\Bbbk) \cap U_{\kk}$. 
    Prove that $\U_v'(\Bbbk)$ is open in $\Quot_d ^r(\A^n)(\Bbbk)$.
    \item Prove that $\U_v(\Bbbk)$ is open in $\U_v'(\Bbbk)$.
\end{enumerate}
\noindent
Proposition \ref{r:prpkpointsext} gives an open subfunctor $\U \subseteq \Quot_d ^r(\A^n)$ 
represented by the open subscheme $U \subseteq Quot$ such that $U(\kk)=U_{\kk}$. 
We will now sketch the proof that $U$ is of dimension $nd + (r-1)d$. 
Since the scheme $U$ is covered by open subschemes $U_v$, 
we have 
\begin{equation*}
    \dim U = \sup_v \, \{ \dim U_v \} = \dim U_{e_1} \quad \text{for} \ e_1 = (1, 0, \dots, 0).
\end{equation*}
We calculate the dimension of $U_{e_1}$ by showing that it is isomorphic 
to a vector bundle over the subscheme $T \subseteq Hilb$ that parametrizes tuples of points. 
Let $\M = \beta_{*} \Oo_{\mathcal{Z_T}}$ be the locally free $\Oo_T$-module 
associated to the universal element of $T$, as in (\ref{r:equnivt}). 
We define the \emph{vector bundle} associated to $\M$ as
\begin{equation*}
    \mathbb{V}(\M) = \Sspec(\Sym (\M^{\vee})).
\end{equation*}
By \cite[Proposition 11.3]{gw}, for every 
$p \colon Y \rightarrow T$ there is a bijection  
\begin{equation}\label{r:eqspeciso}
    \Hom_{\Sch_{/ T}}(Y, \, \mathbb{V}(\M)) \xrightarrow{\simeq}  \Gamma(Y, \, p^{*} \M) 
\end{equation}
natural in $Y$. 
The module $\M$ is locally free. 
Let $W$ be an open subscheme such that we have the trivialization 
$\M_{\vert _W} \simeq \Oo_W ^{\oplus d}$. 
Then for an open affine $\Spec A \rightarrow W$ the isomorphism (\ref{r:eqspeciso}) gives 
\begin{align*}
    \begin{split}
        \Hom_{\Sch_{/ T}}(\Spec A, \, \mathbb{V}(\M)) & \simeq \Gamma(\Spec A, \, \Oo_{\Spec A})^d \\
        & \simeq \A_T^d(A) \simeq \T(A) \times_{\Bbbk} \A_{\Bbbk}^d (A).
    \end{split}
\end{align*}

\begin{proposition}\label{r:prpue1iso}
    There is a natural isomorphism 
    \begin{equation*}
        \phi \colon \U_{e_1} \rightarrow \Hom_{\Sch_{/T}}(-, \, \mathbb{V}(\M)^{\times r-1}).
    \end{equation*}
\end{proposition}

\begin{proposition}\label{r:prpue1dim}
    The dimension of $U_{e_1} \subseteq \Quot_d ^r(\A^n)$ is $nd+(r-1)d$.
\end{proposition}

\begin{proof}
    By Proposition \ref{r:prpue1iso}, 
    \begin{equation*}
    \dim U_{e_1} = \dim(T \times \A^{d(r-1)}) = \dim T + \dim \A^{d(r-1)} = nd+(r-1)d.
    \end{equation*}
\end{proof}

Now we define the \emph{totally degenerate locus} as corresponding to 
surjections of $S$-modules of the form $(\sfrac{S}{\mm})^{\oplus d}$ for 
a fixed maximal ideal 
\begin{equation*}
    \mm = (x_1, \dots, x_n) \, \lhd \, S = \kk [ x_1, \dots, x_n ]. 
\end{equation*}
An $S$-module $M \simeq (\sfrac{S}{\mm})^{\oplus d}$ can be equivalently characterized 
by 
\begin{equation*}
    \dim_{\kk} M = d \quad \text{and} \quad \mm M = 0. 
\end{equation*}
Let $\Zz_{\mm}(\kk) \subseteq \Quot_d^{r}(\A^n)(\kk)$ denote the subset 
corresponding to the modules of this form. 
We will extend this to a closed subfunctor $\Zz_{\mm}$ of $\Quot_d^r(\A^n)$. 
For every $a \in \mm$ consider the multiplication map $\mu_{a} \colon M \rightarrow M$. 
The condition $\mm M = 0$ is equivalent to 
the vanishing of multiplication maps $\mu_a$ for all $a \in \mm$. 
Consider a more general case. Take a subset $I \subseteq S$ and define 
\begin{align*}
    \Zz_I(\kk)= \{ S^{\oplus r} \twoheadrightarrow M \mid I M = 0 \}_{/ \sim} \subseteq \Quot_d^r(\A^n)(\kk).
\end{align*}
For each $a \in I$ define 
\begin{align*}
    \Zz_a(\kk)= \{ S^{\oplus r} \twoheadrightarrow M \mid a M = 0 \}_{/ \sim} \subseteq \Quot_d^r(\A^n)(\kk).
\end{align*}
then $\Zz_I(\kk) = \bigcap _{a \in I} \Zz_a(\kk)$. 
If the set $\Zz_a(\kk)$ is closed, then so is $\Zz_I(\kk)$. 
\begin{proposition}\label{r:prpzaquot}
    The set $\Zz_a(\kk)$ defines a closed subfunctor $\Zz_a$ of $\Quot_d^r(\A^n)$.
\end{proposition}
\noindent
In order to define the subfunctor $\Zz_a \subseteq \Quot_d^{r}(\A^n)$ 
we use the construction of the functor of zeros from Proposition \ref{r:prpfunctofzeros}. 
We describe the extension of each subset $\Zz_a(\kk)$ to a functor 
of the form $\Zz(\mu)$ for a map $\mu$ associated to $\mu_a$. 
\begin{proof}[Proof of Proposition \ref{r:prpzaquot}]
    Let $[\xi] = [\Oo_{Quot}^{\oplus r} \otimes_{\Bbbk} S \xrightarrowdbl[]{\xi} \M]$ 
    denote the universal element of $\Quot_d^r(\A^n)$, as defined in (\ref{r:equnivquot}). 
    Let the surjection $S^{\oplus r} \twoheadrightarrow M$ correspond to 
    a $\kk$-point $p \colon \Spec \kk \rightarrow Quot$. 
    The multiplication $\mu_a \colon M \rightarrow M$ by $a \in S$ 
    comes from a map $\mu \colon \M \rightarrow \M$ such that 
    \begin{align*}
        \mu _{\vert p} = \mu_a \colon \M_{\vert p} = M  \rightarrow \M_{\vert p} = M, 
    \end{align*}
    and the element $a$ comes from a section $\widetilde{a}$ of $\Oo_{Quot} \otimes S$. 
    On each affine scheme $\Spec B$, the map $\mu \colon \M \rightarrow \M$ 
    restricts to the multiplication map $\mu_{a'} \colon M_B \rightarrow M_B$, 
    where $M_B = \M _{\vert_{\Spec B}}$ is a locally free $S_B$-module and 
    the element $a' \in S_B$ is the restriction of the section $\widetilde{a}$ to $\Spec B$. 
    Then we have 
    \begin{align*}
        \begin{split}
            \Zz(\mu)(B) & \ = \ \{ f \colon \Spec B \rightarrow Quot \mid f^* \mu \colon f^* \M \rightarrow f^* \M \text{ is the zero map} \} \\
            & \ = \ \{ f \colon \Spec B \rightarrow Quot \mid a' M_B = 0 \}.
        \end{split}
    \end{align*}
    \noindent
    This extends the set $\Zz_{a}(\kk)$ to a representable functor. 
    Hence $\Zz_a$ arises as the functor of zeros $\Zz(\mu)$, 
    and so it forms a closed subfunctor of $\Quot_d^r(\A^n)$ by Proposition \ref{r:prpfunctofzeros}. 
\end{proof}
\noindent
It follows that for an arbitrary subset $I \subseteq S$ we have $\Zz_I = \bigcap_{a \in I} \Zz_a$ 
and in particular for $I = \mm$ this yields the desired representable functor $\Zz_{\mm}$. 
Now we calculate the dimension of $\Zz_{\mm} \subseteq \Quot_d^r(\A^n)$. 
We define the Grassmannian functor 
$\Gr(d,r) \colon \Sch_{\kk}^{op} \rightarrow \Set$ on $\kk$-points as 
\begin{equation*}
    \Gr(d,r)(\kk) = \left\{ \text{quotients of $\kk$-vector spaces } \kk^{\oplus r} \twoheadrightarrow V \ \vert \ \dim_{\kk} V =d \right\}_{/ \sim}.
\end{equation*}
For an affine scheme $\Spec A$ we have 
\begin{equation*}
    \Gr(d,r)(A) = \left\{ \text{quotients of locally free $A$-modules } A^{\oplus r} \twoheadrightarrow V \ \vert \ \text{$V$ is of rank $d$} \right\}_{/ \sim}.
\end{equation*}
The Grassmannian functor $\Gr(d,r)$ is 
representable by a $\kk$-variety of dimension $(r-d)d$. 
The maximal ideal $\mm \lhd S$ corresponds to a $\kk$-point of $\A^n$, which gives the following isomorphism: 
\begin{equation*}
    \alpha_{\mm} \colon \kk \xhookrightarrow{\iota_{\mm}} S \xrightarrowdbl{\pi_{\mm}} \sfrac{S}{\mm}.
\end{equation*}
We will use this isomorphism to relate $\Gr(d,r)$ and $\Zz_{\mm}$. 
\begin{proposition}\label{r:prpgrzmiso}
    There is a natural isomorphism $\phi \colon \Gr(d,r) \rightarrow \Zz_{\mm}$.
\end{proposition}

\begin{proof}
    We describe the desired map on $\kk$-points, the proof for $\Spec A$ is similar. 
    Take a $\kk$-point 
    \begin{equation*}
        [S^{\oplus r} \xrightarrowdbl[]{\pi} M \simeq (\sfrac{S}{\mm})^{\oplus d}] \in \Zz_{\mm}(\kk).
    \end{equation*}
    Since $\mm M = 0$, the map $\pi$ factors through the inclusion $\iota_{\mm}$. 
    In other words, the map $\pi$ is given by a matrix $(a_{ij})_{i,j} \in \mathbb{M}_{d \times r}(\kk)$, where 
    \begin{equation*}
        \kk^{\oplus r} = \oplus_{i} \, \kk e_i \xhookrightarrow{\iota_{\mm}} S^{\oplus r} = \oplus_{i} \, S e_i \xrightarrowdbl[]{\pi} M \xrightarrow{\simeq} (\sfrac{S}{\mm})^{\oplus d} \xrightarrow{\alpha_{\mm}^{-1}} \kk^{\oplus d} = \oplus_{j} \, \kk e_j
    \end{equation*}
    sends $e_i \mapsto a_{1i}e_1 + \ldots + a_{di}e_d$.  
    This defines a $\kk$-linear surjection $\kk^{\oplus r} \twoheadrightarrow \kk^{\oplus d}$ 
    uniquely up to isomorphism, which gives a point in $\Gr(d,r)(\kk)$. 
    To construct the inverse, note that a $\kk$-point $[\kk^{\oplus r} \twoheadrightarrow V]$ 
    defines an $S$-linear map 
    \begin{equation*}
        (\sfrac{S}{\mm})^{\oplus r} \xrightarrow{\alpha_{\mm}^{-1}} \kk^{\oplus r} \twoheadrightarrow V \simeq \kk^{\oplus d} \xrightarrow{\alpha_{\mm}} ( \sfrac{S}{\mm} )^{\oplus d}, 
    \end{equation*}
    which gives a point $[S^{\oplus r} \twoheadrightarrow V]$ in $\Zz_{\mm}(\kk)$ 
    after composing with $\pi_{\mm} \colon S^{\oplus r} \twoheadrightarrow ( \sfrac{S}{\mm} )^{\oplus r}$. 
    This extends to a natural isomorphism 
    $\phi \colon \Gr(d,r) \rightarrow \Zz_{\mm}$. 
\end{proof}
\begin{proposition}\label{r:prpzaquotdim}
    The dimension of the closed subscheme $\Zz_{\mm} \subseteq \Quot_d^r(\A^n)$ is $(r-d)d$.
\end{proposition}
\begin{proof}
    By Proposition \ref{r:prpgrzmiso}, $\dim \Zz_{\mm} = \dim \Gr(d,r) = (r-d)d$. 
\end{proof}

By Lemma \ref{r:lemirreddim}, 
if $\dim U_{e_1} < \dim \Zz_{\mm}$ then $\Quot_d^r(\A^n)$ is reducible. 
By Propositions \ref{r:prpue1dim} and \ref{r:prpzaquotdim}, we deduce that 
this is the case only when $n < 1-d$. 
We discuss the irreducibility of the Quot scheme of two points.

\begin{proposition}\label{r:prpquotirred}
    Let $n \geq 1, r \geq 2$. Then $\Quot_2 ^r (\A^n)$ is irreducible. 
\end{proposition}
\begin{proof}
    Let $[p \colon S^{\oplus r} \twoheadrightarrow M = \sfrac{S^{\oplus r}}{K}] \in \Quot_2 ^r (\A^n)(\kk)$. 
    It suffices to show that every $[p]$ of this form arises as a limit of tuples of points, 
    that is, points of $\U \subseteq \Quot_2 ^r (\A^n)$. 
    Since $\dim _{\kk} M = 2$, $\Supp M$ consists of at most two points 
    corresponding to maximal ideals $\mm_i \lhd S$. 
    If $\Supp M = \{ \mm_1, \mm_2 \}$ with $\mm_1 \neq \mm_2$, 
    then $\dim_{\kk} M_{\mm_1} = \dim_{\kk} M_{\mm_1} =1$, so 
    $M$ corresponds to a tuple of points given by $\{ \mm_1, \mm_2 \}$. 
    If $\Supp M = \{ \mm \}$, then either $\mm M = 0$ or $\mm^2 M = 0$.
    
    If $\mm M = 0$, then $[p]$ defines a point in the locus $\Zz_{\mm}$, 
    so $M \simeq (\sfrac{S}{\mm})^{\oplus 2}$. Without loss of generality assume 
    $\mm = (x_1, \ldots, x_n)$. Let $e_1, \ldots, e_r$ be the basis of $S^{\oplus r}$ 
    such that each $e_i$ corresponds to $1 \in S$. 
    Let $\varepsilon_1, \varepsilon_2$ be the basis of $M$ 
    with each $\varepsilon_i$ corresponding to $1 \in \sfrac{S}{\mm}$. 
    From Proposition \ref{r:prpgrzmiso} we have the isomorphism 
    $\Gr(2,r) \simeq \Zz_{\mm}$,
    so $[p \colon S^{\oplus r} \twoheadrightarrow M = (\sfrac{S}{\mm})^{\oplus 2}] \in \Zz_{\mm}(\kk)$ 
    is determined by a $2 \times r$ full rank matrix in bases 
    $e_1, \ldots, e_r$ of $S^{\oplus r}$,  
    and $\varepsilon_1, \varepsilon_2$ of $M$. 
    Consider 
    \begin{equation*}
        p(t) \colon S^{\oplus r} \twoheadrightarrow M_t \simeq 
        \sfrac{S}{\mm} \oplus \sfrac{S}{\mm_t} = \sfrac{S}{\mm} \oplus \sfrac{S}{(x_1 -t, x_2, \ldots, x_n)}, 
    \end{equation*}
    given by the same matrix as $p$, 
    where $M_t$ is equipped with the basis $\widetilde{\varepsilon}_1, \widetilde{\varepsilon}_2$ 
    corresponding to $1 \in \sfrac{S}{\mm}$ and $1 \in \sfrac{S}{\mm_t}$, respectively. 
    Then 
    \begin{equation*}
        p = \lim\limits_{t \rightarrow 0} \, p(t).
    \end{equation*}

    If $\mm^2 M = 0$, then $M$ is isomorphic to $\sfrac{\kk[x]}{(x^2)}$. 
    Let $\eta_1, \eta_2$ be the basis corresponding to $1, x \in \sfrac{\kk[x]}{(x^2)}$. 
    We will show that this arises as a flat limit of $\sfrac{S}{I_t}$, 
    where $I_t = \mm \cap \mm_t = \mm \cdot \mm_t$. 
    The ideal $I_t$ contains elements of the forms
    \begin{equation*}
        x_i (x_1-t) \quad \text{and} \quad x_i \ \text{ for } i \neq 1,
    \end{equation*}
    so $x_2, \ldots, x_n \in I_t$ and in $\sfrac{S}{I_t}$ we have $x_1 x_i \equiv t x_i$, that is,
    \begin{equation*}
        x_1^2 \equiv t x_1 \quad \text{and} \quad x_1 x_i \equiv 0 \ \text{ for } i \neq 1.
    \end{equation*}
    Letting $t \rightarrow 0$, we get $I_0 = \lim\limits_{t \rightarrow 0} \, I_t$ 
    such that 
    \begin{equation*}
        x_1^2 \equiv 0, \ x_1 x_i \equiv 0 \ \text{ and } \ x_2, \ldots, x_n \in I_0.
    \end{equation*}
    Hence $\sfrac{S}{I_0} \simeq \sfrac{\kk[x]}{(x^2)}$. 
    To construct the limit 
    \begin{equation*}
        [p \colon S^{\oplus r} \twoheadrightarrow M \simeq \sfrac{\kk[x]}{(x^2)}] = 
        \lim\limits_{t \rightarrow 0} \, [p(t) \colon S^{\oplus r} \twoheadrightarrow M \simeq \sfrac{S}{I_t}],
    \end{equation*}
    let $\widetilde{\eta}_1, \widetilde{\eta}_2$ be the basis corresponding to $1, x_1 \in \sfrac{S}{I_t}$ 
    and define $p(t)$ by the same matrix as we want $p$ to be, 
    in bases $\{ \widetilde{\eta}_i \}$ and $\{ \eta_i \}$, respectively. 
\end{proof}

\subsection{Tensors}\label{r:sectensors}

We review the notions from the theory of tensors. 
Fix $\kk$-vector spaces $A,B,C$ of finite dimension $d$. 
A \emph{tensor} is an element of $A \otimes B \otimes C$.

\begin{definition}
    Let $\tau \in A \otimes B \otimes C$. 
    We say that $\tau$ is \emph{$C$-concise} if the corresponding map 
    $A^{\vee} \otimes B^{\vee} \rightarrow C$ is surjective. 
    Equivalently, is there is no proper subspace $C' \subseteq C$ 
    such that $\tau \in A \otimes B \otimes C'$. 
    We say that $\tau$ is \emph{concise} if it is concise on each coordinate. 
\end{definition}

\begin{definition}
    Two tensors $\tau, \tau' \in A \otimes B \otimes C$ are isomorphic if 
    they are in the same orbit of $\GL (A) \times \GL (B) \times \GL (C) $ 
    acting by change of basis.  
\end{definition}

\begin{definition}
    Let $\tau \in A \otimes B \otimes C$. 
    We say that $\tau$ is of \emph{rank} 1 if $\tau = a \otimes b \otimes c$ 
    for some $a \in A, b \in B, c \in C$. 
    We say that $\tau$ is of \emph{rank $r$} if $r$ is minimal such that 
    $\tau$ can be written as a sum of $r$ rank 1 tensors; 
    and of \emph{border rank $r$} if $r$ is minimal such that 
    $\tau$ can be represented as a limit of $r$ tensors of rank 1.
\end{definition}
\noindent
The notion of border rank has a geometric interpretation in terms of secant varieties. 
\begin{definition}\label{r:defsecant}
    Let $V$ be a vector space over $\kk$ and let $X \subseteq \PP V$ 
    be a projective variety. For $r \in \Z_{+}$ define the locus of 
    the $r$-th secant planes 
    \begin{equation*}
        \sigma_r ^{\circ} (X) = \bigcup\limits_{p_i \in X} \langle p_1, \ldots, p_r \rangle \subseteq \PP V, 
    \end{equation*}
    and define the \emph{$r$-th secant variety of $X$} as its closure 
    \begin{equation*}
        \sigma_r (X) = \overline{\sigma_r ^{\circ} (X)} \subseteq \PP V.
    \end{equation*}
\end{definition}
\noindent
Let $X = \PP (A) \times \PP (B) \times \PP (C)$ and consider 
the Segre embedding 
\begin{equation*}
    X \hookrightarrow \PP(A \otimes B \otimes C), \quad ( [a], [b], [c] ) \mapsto [a \otimes b \otimes c].
\end{equation*}
Then the affine cone $\widehat{\sigma}_r$ over $\sigma_r := \sigma_r(X)$ parametrizes 
tensors of border rank at most $r$. 
We can moreover describe the Segre secant variety in terms of concise tensors. 

\begin{definition}\label{r:defminbrk}
    Let $\tau \in A \otimes B \otimes C$ with 
    $\dim_{\kk} A = \dim_{\kk} B = \dim_{\kk} C = d < \infty$. 
    We say that $\tau$ is of \emph{minimal border rank} if it is concise and of border rank $d$. 
\end{definition}
\noindent
We can also relate the secant varieties to the action of $\GL (A) \times \GL (B) \times \GL (C)$. 
\begin{definition}\label{r:defunittensor}
    Let $A,B,C$ be $\kk$-vector spaces with 
    $\dim_{\kk} A = \dim_{\kk} B = \dim_{\kk} C = d < \infty$. 
    Fix bases $\{ a_i \}, \{ b_i \}, \{ c_i \}$ of $A,B,C$, respectively. 
    A tensor of the form 
    \begin{equation*}
        \mu_d = a_1 \otimes b_1 \otimes c_1 + \ldots + a_d \otimes b_d \otimes c_d
    \end{equation*}
    is called a \emph{unit tensor}. 
\end{definition}
\noindent
Note that unit tensors are concise of rank $d$.
Moreover, all unit tensors are isomorphic, 
so we can take any tensor of this form and call it the unit tensor.

\begin{proposition}\label{r:prpunitorb}
    The variety $\widehat{\sigma}_r$ coincides with the closure of the orbit 
    of the unit tensor under the action of $\GL (A) \times \GL (B) \times \GL (C) $.
\end{proposition}
\noindent
In other words, the closure of the set of all tensors of rank at most $r$ 
is the same as the closure of the set of all concise tensors of rank at most $r$. 

Now we review some results that will be useful when we discuss the irreducibility of $\Bilin_{d,d,d}^{r_1, r_2}(\A^n)$. 
\begin{lemma}\label{r:lemconopen}
    The subspace of concise tensors in $\PP(A \otimes B \otimes C)$ is open.
\end{lemma}
\begin{proof}
    By definition
    \begin{equation*}
        \left\{ \tau \in A \otimes B \otimes C \mid \tau \text{ is concise} \right\} = 
        \{ \tau \text{ is $A$-concise} \} \cap 
        \{ \tau \text{ is $B$-concise} \} \cap 
        \{ \tau \text{ is $C$-concise} \},
    \end{equation*}
    so it suffices to prove that $\left\{ \tau \text{ is $C$-concise} \right\}$ is open. 
    Again, by definition we have 
    \begin{equation*}
        \{ \tau \in A \otimes B \otimes C \mid \tau \text{ is $C$-concise} \} = 
        \{ \text{the induces map $\tau^{\vee} \colon A^{\vee} \otimes B^{\vee} \rightarrow C$ is surjective} \}.
    \end{equation*}
    The map $\tau^{\vee} \colon A^{\vee} \otimes B^{\vee} \rightarrow C$ is given by a $d^2 \times d$ matrix 
    and surjectivity is equivalent to this matrix having a non-zero minor. 
\end{proof}

\begin{proposition}\label{r:prpsecantdim}
    Let $X \subseteq \PP ^N$ be an irreducible projective variety. Then 
    \begin{equation*}
        \dim \sigma_r(X) \leq \min \, \{ N, \, r ( \dim X +1) -1 \}.
    \end{equation*}
\end{proposition}

\begin{proof}
    Follows directly from the definition of expected dimension \cite[Definition 2.7]{carlini}.
\end{proof}

\section{The Bilinear scheme}\label{r:secbilin}

We have defined the Bilinear functor $\Bilin_{d_1,d_2,d_3}^{r_1, r_2}(\A^n)$ 
as the functor whose $\Bbbk$-points $[p_1,p_2,\pi]$ correspond to surjections of $S$-modules 
\begin{center}
    $p_1 \colon S^{\oplus r_1} \twoheadrightarrow M_1 = \sfrac{S^{\oplus r_1}}{K_1} \quad$ 
    and $\quad p_2 \colon S^{\oplus r_2} \twoheadrightarrow M_2 = \sfrac{S^{\oplus r_2}}{K_2}$,

    $\text{together with } \pi \colon M_1 \otimes_S M_2 \twoheadrightarrow M_3 = \sfrac{S^{\oplus r_1 r_2}}{K_3}$, 
\end{center}
\noindent
such that $\dim_{\Bbbk} M_i = d_i$. 
We identify surjections with equal kernels. 
For an arbitrary $\kk$-scheme $X$ the Bilinear functor is defined as 
\[
  \Bilin_{d_1,d_2,d_3}^{r_1, r_2}(\A^n)(X) = 
  \left\{\begin{minipage}[c]{\widthof{\text{such that $\M_j$ is locally free of rank $d_j$ over $\Oo_X$ for $j=1,2,3$}}}
    \centering
    \text{\text{quasicoherent quotients of $(\Oo_X \otimes_{\kk} S)$-modules }}\\
    \text{$\Oo_X^{\oplus r_i} \otimes S \xrightarrowdbl[]{p_i} \M_i$ for $i=1,2$;}
    \text{together with a quotient $\M_1 \otimes \M_2 \xrightarrowdbl[]{\pi} \M_3$;}\\
    \text{such that $\M_j$ is locally free of rank $d_j$ over $\Oo_X$ for $j=1,2,3$}
  \end{minipage}\right\}_{/ \sim}
\]
where we identify surjections with equal kernels. 
For an affine scheme $X=\Spec A$ we write 
\begin{equation*}
    \Bilin_{d_1,d_2,d_3}^{r_1, r_2}(\A^n)(A) = \Bilin_{d_1,d_2,d_3}^{r_1, r_2}(\A^n)(\Spec A).
\end{equation*}
In that case we have
\[
  \Bilin_{d_1,d_2,d_3}^{r_1, r_2}(\A^n)(A) = 
  \left\{\begin{minipage}[c]{\widthof{\text{such that $M_j$ is locally free of rank $d_j$ over $A$ for $j=1,2,3$}}}
    \centering
    \text{$S_A$-linear surjections $A^{\oplus r_i} \otimes S \xrightarrowdbl[]{p_i} M_i$ for $i=1,2$;}
    \text{together with a surjection $M_1 \otimes_{S_A} M_2 \xrightarrowdbl[]{\pi} M_3$;}\\
    \text{such that $M_j$ is locally free of rank $d_j$ over $A$ for $j=1,2,3$}
  \end{minipage}\right\}_{/ \sim}
\]

\subsection{Relation to the Hilbert and Quot schemes}\label{r:sechilbbilin}

The Bilinear functor can be viewed as a generalization of the Hilbert scheme in the following way.
\begin{proposition}
    Let $d = d_i$ for $i=1,2,3$ and $r_1=r_2=1$. Then there is an isomorphism 
\begin{equation*}
    \Bilin_{d,d,d}^{1, 1}(\A^n) \simeq \Hilb_d(\A^n).
\end{equation*}
\end{proposition}

\begin{proof}
    The $\kk$-points of $\Bilin_{d,d,d}^{1, 1}(\A^n)$ correspond to 
    \begin{center}
        $p_1 \colon S^{\oplus 1} \twoheadrightarrow M_1 = \sfrac{S^{\oplus 1}}{J_1}$ 
        and $p_2 \colon S^{\oplus 1} \twoheadrightarrow M_2=\sfrac{S^{\oplus 1}}{J_2}$, 
        where $J_1$, $J_2$ are ideals in $S$, 

        together with $\pi \colon M_1 \otimes_S M_2 \twoheadrightarrow M_3$.
    \end{center}
    \noindent
    We have 
    \begin{equation}\label{r:eqdimineq}
        \dim_{\kk} \sfrac{S^{\oplus 1}}{J_1} \otimes_S \sfrac{S^{\oplus 1}}{J_2} = \dim_{\kk} \sfrac{S^{\oplus 1}}{(J_1 + J_2)} \leq \dim_{\kk} \sfrac{S^{\oplus 1}}{J_1} = \dim_{\kk} \sfrac{S^{\oplus 1}}{J_2} = \dim_{\kk} M_3.
    \end{equation} 
    The surjection $\pi \colon \sfrac{S^{\oplus 1}}{(J_1 + J_2)} \twoheadrightarrow M_3$ 
    exists only if $\dim_{\kk} \sfrac{S^{\oplus 1}}{(J_1 + J_2)} \geq \dim_{\kk} M_3$, 
    so the inequality in (\ref{r:eqdimineq}) must be an equality and the map 
    $\pi \colon \sfrac{S^{\oplus 1}}{(J_1 + J_2)} \simeq \sfrac{S^{\oplus 1}}{J_1} \twoheadrightarrow M_3$ 
    must be an isomorphism of $S$-modules. 
    Thus, $J_1 = J_2 = J_3$ and since we identify surjections with equal kernels, 
    the point $[p_1,p_2,\pi]$ is uniquely determined by a single map 
    $p_1 = p_2 = p \colon S \twoheadrightarrow M = \sfrac{S}{J}$. 
    This gives the desired isomorphism. 

\end{proof}

The Bilinear functor can also be viewed as a generalization of the Quot scheme. 
There are two ways to see this connection. 
The first point of view has been explained in \cite[Problem XXXVIII]{jj}: 
the functor $\Bilin_{d_1,d_2,d_2}^{r_1, r_2}(\A^n)$ 
gives a space that parametrizes modules together with an algebra acting on them. 
For the second point of view, let $[p_1,p_2,\pi]$ be a $\kk$-point of $\Bilin_{d_1,d_2,d_3}^{r_1, r_2}(\A^n)$ 
and notice that each of the three maps $p_1,p_2, \pi$ determines a $\kk$-point of the Quot scheme. 
For maps $p_i$ the points of the Quot schemes are defined as 
$[p_i \colon S^{\oplus r_i} \twoheadrightarrow M_i = \sfrac{S^{\oplus r_i}}{K_i}] \in \Quot_{d_i}^{r_i}(\A^n)(\kk)$. 
The map $\pi$ has a unique lift $p_3 \colon S^{\oplus r_1 r_2} \twoheadrightarrow M_3 = \sfrac{S^{\oplus r_1 r_2}}{K_3}$, 
which determines a $\kk$-point of $\Quot_{d_3}^{r_1 r_2}(\A^n)$. 
We will show that this extends to a natural transformation and 
use it to generalize the known results on the Quot scheme to 
the setting of the Bilinear functor. 

\subsection{Relation to tensors}\label{r:secbilintensors}
We now describe the special case $\Bilin_{d,d,d}^{r_1, r_2} (\A^n)$ and characterize tensors induced by the $\kk$-points. 
Take $[p_1, p_2, \pi] \in \Bilin_{d,d,d}^{r_1, r_2} (\A^n)(\kk)$ given by
\begin{equation*}
    p_i \colon S^{\oplus r_i} \twoheadrightarrow M_i \ (i=1,2), \quad \pi \colon M_1 \otimes M_2 \twoheadrightarrow M_3,
\end{equation*}
with $\dim_{\kk} M_i = d$. 
Let $\{ a_j \}, \{ b_j \}, \{ c_j \}$ be bases over $\kk$ of $M_1, \, M_2, \, M_3$, respectively. 
From the natural isomorphism of vector spaces $V \otimes W^{\vee} \simeq \Hom (V^{\vee}, \, W)$
we get the tensor corresponding to $\pi$, defined as
\begin{equation*}
    \mu_{\pi} \in M_1 ^{\vee} \otimes M_2 ^{\vee} \otimes M_3.
\end{equation*}
If $\pi$ is given by $\pi (a_i \otimes b_j) = \sum_{1 \leq k \leq d} s_{kij} c_k$, then 
\begin{equation*}
    \mu_{\pi} = \sum\limits_{i, j} a_i^{*} \otimes b_j^{*} \otimes \big( \sum\limits_{k} s_{kij} c_k \big),
\end{equation*}
where $\{ a_j^{*} \}, \{ b_j^{*} \}$ are dual bases of $M_1^{\vee}, \, M_2^{\vee}$, respectively.
Note that the maps $\pi \colon M_1 \otimes M_2 \twoheadrightarrow M_3$ coming from 
points of $\Bilin_{d,d,d}^{r_1, r_2} (\A^n)$ correspond to $M_3$-concise tensors in 
$M_1^{\vee} \otimes M_2^{\vee} \otimes M_3$. 
\begin{proposition}
    Let $[p_1, p_2, \pi] \in \Bilin_{d,d,d}^{r_1, r_2} (\A^n)$. 
    Then $\pi$ corresponds to a tensor of rank at least $d$.
\end{proposition}
\begin{proof}
    Since $\pi$ is surjective, the corresponding tensor is $M_3$-concise 
    and also $\dim _{\kk} M_1 \otimes M_2 \geq d$. 
    A contrario, suppose that the corresponding tensor $\mu$ is of rank $r <d$. 
    Then there are bases $\{ a_i \}, \, \{ b_i \}$ of $M_1, \, M_2$ such that 
    \begin{equation*}
        \mu = a_1^{*} \otimes b_1^{*} \otimes \pi(a_1 \otimes b_1) + \ldots + a_r^{*} \otimes b_r^{*} \otimes \pi(a_r \otimes b_r),
    \end{equation*}
    which implies that the images $\pi (a_i \otimes b_i)$ span 
    a subspace of dimension $r<d$ in $M_3$, contra surjectivity of $\pi$.
\end{proof}

We will later see that the locus corresponding to concise tensors is open and 
gives rise to the main irreducible component of $\Bilin_{d,d,d}^{r_1, r_2} (\A^n)$.

\subsection{The Bilinear functor is representable}\label{r:secbilinisrepr}
The Bilinear functor $\Bilin_{d_1,d_2,d_3}^{r_1, r_2} (\A^n)$ 
is a closed subfunctor of a representable functor, 
hence is representable. 
Each of the surjections $p_i$ gives a $\Bbbk$-point of the Quot scheme 
$\Quot_{d_i}^{r_i}(\A^n)$ for $i=1,2$, 
and the third map $\pi$ lifts uniquely to a surjection $S^{\oplus r_1 r_2} \twoheadrightarrow M_3$, 
so it gives a $\Bbbk$-point of the Quot scheme $\Quot_{d_3}^{r_1 r_2}(\A^n)$. 
This defines the following inclusion on $\Bbbk$-points:
\begin{equation*}
    \Bilin_{d_1,d_2,d_3}^{r_1, r_2}(\A^n)(\Bbbk) \hookrightarrow (\Quot_{d_1}^{r_1}(\A^n) \times \Quot_{d_2}^{r_2}(\A^n) \times \Quot_{d_3}^{r_1 r_2}(\A^n))(\Bbbk).
\end{equation*}
The same reasoning defines the inclusion  
\begin{equation*}
    \Bilin_{d_1,d_2,d_3}^{r_1, r_2}(\A^n)(A) \hookrightarrow (\Quot_{d_1}^{r_1}(\A^n) \times \Quot_{d_2}^{r_2}(\A^n) \times \Quot_{d_3}^{r_1 r_2}(\A^n))(A)
\end{equation*}
for an arbitrary affine scheme $\Spec A$, 
which shows that $\Bilin_{d_1,d_2,d_3}^{r_1, r_2}(\A^n)$ is a subfunctor of the product  
$\Quot_{d_1}^{r_1}(\A^n) \times \Quot_{d_2}^{r_2}(\A^n) \times \Quot_{d_3}^{r_1 r_2}(\A^n)$. 
Throughout this section we will denote the product 
$\Quot_{d_1}^{r_1}(\A^n) \times \Quot_{d_2}^{r_2}(\A^n) \times \Quot_{d_3}^{r_1 r_2}(\A^n)$ by $\Q$. 
We will use Proposition \ref{r:prpclsubfcriterion} to prove 
that $\Bilin_{d_1,d_2,d_3}^{r_1, r_2}(\A^n)$ is a closed subfunctor of $\Q$. 

We start by describing the fiber product $\underline{T} \times_{\Q} \Bilin_{d_1,d_2,d_3}^{r_1, r_2}(\A^n)$ 
from Definition \ref{r:defopclsubf} to gain insight into the properties of 
the ideal satisfying Proposition \ref{r:prpclsubfcriterion}, provided that it exists. 
Let $T=\Spec A$, $T'=\Spec B$ be affine schemes. 
It suffices to consider the case $B = \sfrac{A}{J}$, where $J \lhd A$, 
since we always have factorization through the kernel 
$f^{\#} \colon A \twoheadrightarrow \sfrac{A}{\ker(f^{\#})} \rightarrow B$. 
Let $[p_i]$ denote an element of $\Q(A)$ defined as 
\begin{align*}
    [p_i] & = [p_1,p_2,p_3]  \\
    & = \big[ S_A^{\oplus r_1} \xrightarrowdbl[]{p_1} M_1, \, S_A^{\oplus r_2} \xrightarrowdbl[]{p_2} M_2, \, S_A^{\oplus r_1 r_2} \xrightarrowdbl[]{p_3} M_3 \big]. 
\end{align*}
\noindent
It corresponds to an element of $\Bilin_{d_1,d_2,d_3}^{r_1, r_2}(\A^n)(A)$ if and only if the third map $S_A^{\oplus r_1 r_2} \xrightarrowdbl[]{p_3} M_3$ factors as 
\[\begin{tikzcd}
	{S_A^{\oplus r_1 r_2}} \\
	\\
	{\frac{S_A^{\oplus r_1 r_2}}{K_1 \otimes S_A^{\oplus r_2}+S_A^{\oplus r_1} \otimes K_2} \simeq M_1 \otimes M_2} &&& {M_3=\frac{S_A^{\oplus r_1 r_2}}{K_3}}
	\arrow[two heads, from=1-1, to=3-1]
	\arrow["{p_3}", two heads, from=1-1, to=3-4]
	\arrow["\pi", two heads, from=3-1, to=3-4]
\end{tikzcd}\]
\noindent
where $M_i = \sfrac{S^{\oplus r_i}}{K_i}$ for $i=1,2$. 
If the map $\pi$ exists, 
then it is unique, so in that case we can write $[p_i] \in \Bilin_{d_1,d_2,d_3}^{r_1, r_2}(\A^n)(A)$. 
On the other hand, the map $\pi \colon M_1 \otimes M_2 \twoheadrightarrow M_3 = \sfrac{S^{\oplus r_1 r_2}}{K_3}$ 
exists if and only if both compositions  
\begin{equation*}
    K_1 \otimes S_A^{\oplus r_2} \hookrightarrow S_A^{\oplus r_1 r_2} \twoheadrightarrow M_3 
    \quad \text{and} \quad
    S_A^{\oplus r_1} \otimes K_2 \hookrightarrow S_A^{\oplus r_1 r_2} \twoheadrightarrow M_3 
\end{equation*}
are zero, which is quivalent to the inclusion of kernels $K_1 \otimes S_A^{\oplus r_2} + S_A^{\oplus r_1} \otimes K_2 \subseteq K_3$. 
This reasoning determines the conditions for  
a pair $(T' \xrightarrow{f} T, \ [p_1', p_2', \pi] \in \Bilin_{d_1,d_2,d_3}^{r_1, r_2}(\A^n)(T'))$
to define an element of the fiber product of functors
\begin{multline*}
    (\underline{T} \times_{\Q} \Bilin _{d_1,d_2,d_3}^{r_1, r_2}(\A^n))(T') = \\
    \bigl\{ (T' \xrightarrow{f} T, \ [p_1', p_2', \pi]) \in \underline{T}(T') \times \Bilin_{d_1,d_2,d_3}^{r_1, r_2}(\A^n)(T') \mid f^{*}[p_i] = \iota([p_1', p_2', \pi]) \bigr\}, 
\end{multline*}
\noindent
where $\iota \colon \Bilin_{d_1,d_2,d_3}^{r_1, r_2}(\A^n) \hookrightarrow \Q$ is the inclusion of functors. 
We write $M_i' = M_i \otimes_A B$ for $i=1,2,3$. 
Then for elements of the form
\begin{align*}
    \begin{split}
        f^{*}[p_i] & = \big[S_B^{\oplus r_1} \xrightarrowdbl[]{p_1'} M_1', \, S_B^{\oplus r_2} \xrightarrowdbl[]{p_2'} M_2', \, S_B^{\oplus r_1 r_2} \xrightarrowdbl[]{p_3'} M_3' \big], \\
        [p_1', p_2', \pi] & = \big[ S_B^{\oplus r_1} \xrightarrowdbl[]{p_1'} M_1', \, S_B^{\oplus r_2} \xrightarrowdbl[]{p_2'} M_2', \, M_1' \otimes_{S_B} M_2' \xrightarrowdbl[]{\pi} M_3' \big] 
    \end{split}
\end{align*}
we have $f^{*}[p_i] = \iota([p_1', p_2', \pi])$ 
if and only if $f^{*}[p_i]$ defines an element of 
${\Bilin_{d_1,d_2,d_3}^{r_1, r_2}(\A^n)(B)}$, 
which is equivalent to the factorization of the map 
$p_3' \colon S^{\oplus r_1 r_2} \twoheadrightarrow M_3'$ 
through $M_1' \otimes_{S_B} M_2' \xrightarrowdbl[]{\pi} M_3'$. 
Let $M_i' = \sfrac{S_B^{\oplus r_i}}{K_i'}$ for $i=1,2$ 
and $M_3' = \sfrac{S_B^{\oplus r_1 r_2}}{K_3'}$. 
Then, as before, the map $\pi$ exists if and only if we have the inclusion of kernels 
\begin{equation}\label{r:eqinclusion}
    K_1' \otimes S_B^{\oplus r_2} + S_B^{\oplus r_1} \otimes K_2' \subseteq K_3'
\end{equation}
\noindent
or equivalently, if the compositions 
\begin{equation}\label{r:eqcompzero}
    K_1' \otimes S_B^{\oplus r_2} \rightarrow S_B^{\oplus r_1 r_2} \twoheadrightarrow M_3' \quad \text{and} \quad S_B^{\oplus r_1} \otimes K_2' \rightarrow S_B^{\oplus r_1 r_2} \twoheadrightarrow M_3'
\end{equation}
are zero. 
In other words, given an element $[p_i] \in \Q(A)$, 
we are looking for quotients $A \twoheadrightarrow \sfrac{A}{J}$ that 
yield the zero maps after taking the tensor products of 
$K_1 \otimes S_A^{\oplus r_2}, \ S_A^{\oplus r_1} \otimes K_2 \twoheadrightarrow M_3$ 
with $\sfrac{A}{J}$. 

On the other hand, claiming that $\Bilin_{d_1,d_2,d_3}^{r_1, r_2}(\A^n)$ is a 
closed subfunctor of $\Q$ is equivalent to 
claiming that for all affine schemes $T$, $T'$, 
the fiber product $\underline{T} \times_{\Q} \Bilin_{d_1,d_2,d_3}^{r_1, r_2}(\A^n)$ is representable 
by a closed subscheme $Z = \Spec(\sfrac{A}{I}) \subseteq T$, 
where $I \lhd A$ is the ideal from Proposition \ref{r:prpclsubfcriterion}. 
Combining those two points of view leads to the conclusion that 
we are looking for the ideal $I \lhd A$ such that the 
map $f \colon \Spec(\sfrac{A}{J}) \rightarrow \Spec A$ 
satisfies the conditions (\ref{r:eqinclusion}) or (\ref{r:eqcompzero}) 
if and only if it factors through $\Spec(\sfrac{A}{I}) \subseteq \Spec A$. 
\begin{condition}[for the ideal $I \lhd A$]\label{r:idealcond}
    Let $f \colon \Spec B \rightarrow \Spec A$, where $B = \sfrac{A}{J}$. 
    Let 
    \begin{equation*}
        [p_i] = \big[ S_A^{\oplus r_1} \xrightarrowdbl[]{p_1} M_1, \, S_A^{\oplus r_2} \xrightarrowdbl[]{p_2} M_2, \, S_A^{\oplus r_1 r_2} \xrightarrowdbl[]{p_3} M_3 \big] \in \Q(A)
    \end{equation*}
    be such that 
    \begin{equation*}
        f^{*}[p_i] = \big[ S_B^{\oplus r_1} \xrightarrowdbl[]{p_1'} M_1', \ S_B^{\oplus r_2} \xrightarrowdbl[]{p_2'} M_2', \ S_B^{\oplus r_1 r_2} \xrightarrowdbl[]{p_3'} M_3' \big] \in \Q(B).
    \end{equation*}
    Let $M_i = \sfrac{S_A^{\oplus r_i}}{K_i}$ and $M_i' = \sfrac{S_B^{\oplus r_i}}{K_i'}$ for $i=1,2$. 
    Define maps 
    \begin{align*}
        & \psi_1 \colon K_1' \otimes S_B^{\oplus r_2} = K_1 \otimes S_A^{\oplus r_2} \otimes \sfrac{A}{J} \rightarrow S_B^{\oplus r_1 r_2} \rightarrow M_3', \\
        & \psi_2 \colon S_B^{\oplus r_1} \otimes K_2' = S_A^{\oplus r_1} \otimes K_2 \otimes \sfrac{A}{J} \rightarrow S_B^{\oplus r_1 r_2} \rightarrow M_3'
    \end{align*}
    \noindent
    as compositions of the canonical inclusions and quotient maps. 
    In this setting, we get zero maps $\psi_i \equiv 0$ if and only if $I \subseteq J$. 
\end{condition}

\begin{theorem}\label{r:thmexists}
    The ideal satisfying Condition \ref{r:idealcond} exists.
\end{theorem}
\noindent
Before proving Theorem \ref{r:thmexists}, 
we describe the motivation for what the ideal $I$ should be. 
Let 
\begin{equation*}
    Q_1 = K_1 \otimes S_A^{\oplus r_2} \quad \text{and} \quad Q_2 = S_A^{\oplus r_1} \otimes K_2.
\end{equation*}
Consider the compositions $\psi_j \colon Q_j \rightarrow S_A^{\oplus r_1 r_2} \twoheadrightarrow M_3$ for $j=1,2$. 
Those are morphisms of locally free $A$-modules, 
so they determine morphisms of quasicoherent sheaves on $\Spec A$. 
Consider restrictions to an open subscheme 
$\Spec R = \Spec A_f \subseteq \Spec A$ such that the modules $Q_i$ and $M_3$ are free. 
Then we can write 
\begin{equation*}
    \varphi_i = (\psi_i)_{\vert _{\Spec R}} \colon M = (Q_i)_{\vert _{\Spec R}} \rightarrow N = (M_3)_{\vert _{\Spec R}}
\end{equation*}
for the restriction. 
Those are morphisms of free $R$-modules, so they are given by some matrices $P_i$ with entries in $R$. 
Consider a morphism $f \colon \Spec(\sfrac{R}{J}) \rightarrow \Spec R$ 
corresponding to the quotient by an ideal $J \lhd R$. 
This induces the pullback map $f^{*} \colon \Q(R) \rightarrow \Q(\sfrac{R}{J})$ 
and sends the maps $\varphi_i$ to 
\begin{equation*}
    \varphi_i ' = \varphi_i \otimes 1 \colon M_i' = M_i \otimes _B \sfrac{R}{J} \rightarrow N' = N \otimes_B \sfrac{R}{J}.
\end{equation*}
Since the modules $M, N$ are free, we have the isomorphisms 
\begin{align*}
    M_i'& = M_i \otimes \sfrac{R}{J} \simeq R^{\oplus m_i} \otimes \sfrac{R}{J} \simeq (\sfrac{R}{J})^{\oplus m_i} \text{ for } i=1,2,  \\
    N' & = N \otimes \sfrac{R}{J} \simeq R^{\oplus n} \otimes \sfrac{R}{J} \simeq (\sfrac{R}{J})^{\oplus n},
\end{align*}
which shows that $M_i', \ N'$ are also free. 
The maps $\varphi_i'$ are given by matrices $P_i'$ whose entries 
are the images of the entries of $P_i$ in $\sfrac{R}{J}$. 
It follows that the maps $\varphi_i'$ are zero if and only if 
the entries of $P_i$ are in the ideal $J \lhd R$. 
This suggests that the ideal $I$ from Condition \ref{r:idealcond} 
must contain the entries of matrices $P_i$ of this form, 
or in other words, it must contain the Fitting ideal from Definition \ref{r:deffittingideal}. 
\begin{proof}[Proof of Theorem \ref{r:thmexists}]
    Let $Q_1 = K_1 \otimes S_A^{\oplus r_2}$ and $Q_2 = S_A^{\oplus r_1} \otimes K_2$ be as above. 
    Consider the compositions 
    \begin{equation*}
        \psi_j \colon Q_j \rightarrow S_A^{\oplus r_1 r_2} \twoheadrightarrow M_3 \quad \text{for } j=1,2.
    \end{equation*} 
    We claim that the construction of Fitting ideal $\I(\psi_j)$ for each $j=1,2$ 
    gives the ideal $\I = \I(\psi_1) + \I(\psi_2)$ that satisfies Condition \ref{r:idealcond}. 
    Take an affine open subscheme $\Spec A_f$ such that 
    the modules $(Q_j) _{\vert_{\Spec A_f}}$ and $(M_3)_{\vert_{\Spec A_f}}$ are free. 
    Then from the isomorphisms in (\ref{r:eqfitting}), 
    the pullbacks of maps $\psi_j$ are zero if and only if $\I \subseteq J$. 
    By Proposition \ref{r:prpfunctofzeros}, this gives a well-defined ideal in $A$. 
\end{proof}
\noindent
The schemes representing functors $\Zz(\psi_1) \cap \Zz(\psi_2)$ defined on affine schemes glue 
to give the closed subscheme $Bilin \subseteq \Q$ 
that represents the Bilinear functor $\Bilin_{d_1,d_2,d_3}^{r_1, r_2}(\A^n)$. 

\subsection{The tangent space to the Bilinear scheme}\label{r:secbilintangent}

Let $ \Bilin_{d_1,d_2,d_3}^{r_1, r_2}(\A^n) \xhookrightarrow{\iota} \Q$ 
denote the closed embedding of functors from Section \ref{r:secbilinisrepr}. 
We will use this result together with Theorem \ref{r:thmquottangent} to 
calculate the tangent space to the Bilinear scheme. 
We start by fixing the notation. 
Let $M_i = \sfrac{F_i}{K_i}$ and $\tM_i = \sfrac{\tF_i}{\tK_i}$ for $i=1,2,3$. 
Write $F_j = S^{\oplus r_j}$, $\tF_j = S[\varepsilon]^{\oplus r_j}$ for $j=1,2$ and 
$\tF_3 = S[\varepsilon]^{\oplus r_1 r_2}$, $F_3 = S^{\oplus r_1 r_2}$.  
Consider 
\begin{align*}
    \begin{split}
        \big[ F_1 \xrightarrowdbl[]{p_1} M_1, \, F_2 \xrightarrowdbl[]{p_2} M_2, \, F_3 \xrightarrowdbl[]{p_3} M_3 \big] & = [p_1, p_2, p_3] \in \Q(\kk), \\
        \big[ \tF_1 \xrightarrowdbl[]{\widetilde{p_1}} \tM_1, \, \tF_2 \xrightarrowdbl[]{\widetilde{p_2}} \tM_2, \, \tF_3 \xrightarrowdbl[]{\widetilde{p_3}} \tM_3 \big] 
        & = [\widetilde{p_1}, \widetilde{p_2}, \widetilde{p_3}] \in \Q(\keps).
    \end{split}
\end{align*}
The point $[\widetilde{p_1}, \widetilde{p_2}, \widetilde{p_3}]$ 
determines a tangent vector to $\Q$ at $[p_1,p_2,p_3]$ if $M_i \simeq \sfrac{\tM_i}{\varepsilon \tM_i}$. 
Consider also 
\begin{align*}
    \begin{split}
        \big[F_1 \xrightarrowdbl[]{p_1} M_1, \, F_2 \xrightarrowdbl[]{p_2} M_2, \, M_1 \otimes_{S} M_2 \xrightarrowdbl[]{\pi} M_3 \big] 
        & = [p_1, p_2, \pi] \in \Bilin_{d_1,d_2,d_3}^{r_1, r_2}(\A^n)(\kk), \\
        \big[\tF_1 \xrightarrowdbl[]{\widetilde{p_1}} \tM_1, \ \tF_2 \xrightarrowdbl[]{\widetilde{p_2}} \tM_2, \ \tM_1 \otimes_{\Seps} \tM_2 \xrightarrowdbl[]{\widetilde{\pi}} \tM_3 \big]
        & = [\widetilde{p_1}, \widetilde{p_2}, \widetilde{\pi}] \in \Bilin_{d_1,d_2,d_3}^{r_1, r_2}(\A^n)(\keps). 
    \end{split}
\end{align*}
As in the setting of $\Q$, the point $[\widetilde{p_1}, \widetilde{p_2}, \widetilde{\pi}]$ 
determines a tangent vector at $[p_1, p_2, \pi]$ 
if the pullback of $[\widetilde{p_1}, \widetilde{p_2}, \widetilde{\pi}]$ along 
$\alpha \colon \Spec \keps \rightarrow \Spec \kk$ gives $[p_1, p_2, \pi]$, 
which is equivalent to $M_i = \sfrac{\tM_i}{\varepsilon \tM_i}$ for each $i=1,2,3$. 

By the same reasoning as in Section \ref{r:secbilinisrepr}, $[p_1, p_2, p_3] \in \Q(\Bbbk)$ 
determines a $\Bbbk$-point $[p_1,p_2,\pi] \in \Bilin_{d_1,d_2,d_3}^{r_1, r_2}(\A^n)(\Bbbk)$ 
if and only if we have the following factorization 
\[\begin{tikzcd}
	{F_3} \\
	{M_1 \otimes M_2} && {M_3}
	\arrow[two heads, from=1-1, to=2-1]
	\arrow["{p_3}", shorten >=6pt, two heads, from=1-1, to=2-3]
	\arrow["\pi", shorten <=5pt, shorten >=5pt, two heads, from=2-1, to=2-3]
\end{tikzcd}\]
\noindent
Denote the kernel of $F_3 \twoheadrightarrow M_1 \otimes M_2$ by 
$Q = K_1 \otimes F_2 + F_1 \otimes K_2 = Q_1 + Q_2$. 
Then we have the isomorphism $M_1 \otimes M_2 \simeq \sfrac{F_3}{Q}$ and 
the map $\pi \colon M_1 \otimes M_2 \twoheadrightarrow M_3 = \sfrac{F_3}{K_3}$ 
exists if and only if the compositions 
$Q_j \hookrightarrow F_3 \twoheadrightarrow M_3$ vanish for $j=1,2$. 
Equivalently the map $\pi$ exists if and only if there is the inclusion $Q \subseteq K_3$. 

Consider the tangent space to $\mathcal{Q}$ at $[p_1,p_2,p_3]$, 
which by Theorem \ref{r:thmquottangent} is isomorphic to 
\begin{equation*}
    \Hom _S (K_1, M_1) \times \Hom _S (K_2, M_2) \times \Hom _S (K_3, M_3).    
\end{equation*}
Take an element 
\begin{equation*}
    (\varphi_1, \varphi_2, \varphi_3) \in \Hom _S (K_1, M_1) \times \Hom _S (K_2, M_2) \times \Hom _S (K_3, M_3). 
\end{equation*}
The above reasoning gives the following diagrams: 
\[\begin{tikzcd}
	{K_1 \otimes F_2} &&& {M_1 \otimes F_2} && {F_1 \otimes K_2} &&& {F_1 \otimes M_2} \\
	{K_3} &&& {M_3} && {K_3} &&& {M_3}
	\arrow["{\varphi_1 \otimes \mathds{1}_{F_2}}", from=1-1, to=1-4]
	\arrow["{j_1}", hook, from=1-1, to=2-1]
	\arrow["\pi \circ ({\mathds{1}_{M_1} \otimes p_2})", from=1-4, to=2-4]
	\arrow["{\mathds{1}_{F_1} \otimes \varphi_1}", from=1-6, to=1-9]
	\arrow["{j_2}", hook, from=1-6, to=2-6]
	\arrow["\pi \circ ({p_1 \otimes \mathds{1}_{M_1}})", from=1-9, to=2-9]
	\arrow["{\varphi_3}", from=2-1, to=2-4]
	\arrow["{\varphi_3}", from=2-6, to=2-9]
\end{tikzcd}\]
\noindent

\begin{proposition}\label{r:thmbilintangent}
    The tangent space to $\Bilin_{d_1,d_2,d_3}^{r_1, r_2}(\A^n) $ at $[p_1,p_2,\pi]$ consists of those elements 
    \begin{equation*}
        (\varphi_1, \varphi_2, \varphi_3) \in \Hom _S (K_1, M_1) \times \Hom _S (K_2, M_2) \times \Hom _S (K_3, M_3)
    \end{equation*}
    that make the diagrams above commute. 
    In other words, $(\varphi_1, \varphi_2, \varphi_3)$ determines the tangent vector at $[p_1,p_2,\pi]$ if and only if for 
    \begin{equation*}
        k_1 \otimes f_2 \in K_1 \otimes F_2 \quad \text{and} \quad f_1 \otimes k_2 \in F_1 \otimes K_2,
    \end{equation*}
     the images $\varphi_1(k_1) \otimes f_2, \, f_1 \otimes \varphi_2(k_2)$ in $\sfrac{F_3}{K_3}$ satisfy 
    \begin{equation*}
        \varphi_3(k_1 \otimes f_2) = \varphi_1(k_1) \otimes f_2 + K_3 \quad \text{and} \quad
        \varphi_3(f_1 \otimes k_2) = f_1 \otimes \varphi_2(k_2) + K_3.
    \end{equation*}
\end{proposition}

\begin{proof}
    By definition,
    $[\widetilde{p_1}, \widetilde{p_2}, \widetilde{\pi}] \in \Bilin_{d_1,d_2,d_3}^{r_1, r_2}(\A^n)(\keps)$ 
    determines a tangent vector at $[p_1,p_2,\pi]$ 
    if and only if it restricts to $[p_1,p_2,\pi]$ when we take the pullback by 
    $\alpha \colon \Spec \kk \rightarrow \Spec \keps$. 
    For $i=1,2$, the maps $\widetilde{p_i}$ restrict to the desired maps $p_i$ 
    if and only if they correspond to tangent vectors to the Quot schemes $\Quot_{d_i}^{r_i}(\A^n)$. 
    In this case, we have the corresponding $S$-linear maps $\varphi_i$. 
    The map $\widetilde{\pi}$ lifts to $\widetilde{p_3}$, 
    which determines a tangent vector to $\Quot_{d_3}^{r_1 r_2}(\A^n)$ at $[p_3]$ if and only if 
    it restricts to $p_3$. If this is the case, it gives an $S$-linear map $\varphi_3$. 
    Let $[\widetilde{p_1}, \widetilde{p_2}, \widetilde{p_3}] \in \Q(\keps)$ restrict to $[p_1,p_2,p_3] \in \Q(\kk)$ 
    such that $p_3$ factors through $\pi$. 
    We will show that $\widetilde{p_3}$ factors through $\widetilde{\pi}$ if and only if the maps $\varphi_i$ 
    satisfy the desired commutative diagrams. 
    This will show that
    $[\widetilde{p_1}, \widetilde{p_2}, \widetilde{\pi}]$ 
    of this form determines a tangent vector at $[p_1,p_2,\pi]$. 

    If we construct the maps $\varphi_i$ for the point $[\widetilde{p_1}, \widetilde{p_2}, \widetilde{p_3}] \in \Q(\keps)$, 
    where $\widetilde{p_3}$ is defined as the lift of $\widetilde{\pi}$, 
    then they will automatically satisfy the desired diagrams. It suffices to prove the converse.  
    Assume that the maps $\varphi_i$ satisfy the diagrams. 
    Recall the module $\tK$ we have defined in (\ref{r:eqtK}). 
    Likewise, define 
    \begin{equation*}
        \tK_i = \{ k+ \varepsilon g \ \vert \ k \in K_i, \ g \in F_i \text{ such that } g+K_i = \varphi_i(k) \}
    \end{equation*}
    for $i=1,2,3$. Let $\widetilde{Q}_1 = \tK_1 \otimes \tF_2$, 
    we will show that $\tQ_1 \subseteq \tK_3$. 
    Elements of $\tQ_1$ are of the following form 
    \begin{equation*}
        (k + \varepsilon g) \otimes (f_1 + \varepsilon f_2) = k \otimes f_1 + \varepsilon g \otimes f_1 + k \otimes \varepsilon f_2.
    \end{equation*}
    \noindent
    We note the following observations.  
    \begin{enumerate}
        \item $k \otimes f_1 \in K_1 \otimes F_2, \text{ hence } k \otimes f_1 + \varepsilon g_3 \in \tK_3 \text{ with } g_3 + K_3 = \varphi_3(k \otimes f_1) = \varphi_1(k) \otimes f_1 +K_3$.
        \item $\varepsilon g \otimes f_1 \text{ is such that } g + K_1 = \varphi_1(k) \text{, so } g \otimes f_1 + K_1 \otimes F_2 = \varphi_1(k) \otimes f_1$.
        \item $k \otimes \varepsilon f_2 \in \varepsilon K_1 \otimes F_1 \text{ and } \varepsilon(k \otimes f_2 + \varepsilon h) = \varepsilon k \otimes f_2 \in \tK_3$ 
        for $h$ such that $\varphi_3 (k \otimes f_2) = h + K_3$.
    \end{enumerate}
    \noindent
    It suffices to show that $k \otimes f_1 + \varepsilon g \otimes f_1 \in \tK_3$. 
    In $\sfrac{F_3}{K_3}$, the elements $g \otimes f_1$ and $g_3$ both map to $\varphi_3(k \otimes f_1)$, 
    hence $\varepsilon g \otimes f_1$ corresponds to $\varepsilon g_3$ and therefore 
    $k \otimes f_1 + \varepsilon g \otimes f_1 \in \tK_3$. 
    The same reasoning shows that $\tQ_2 \subseteq \tK_3$. 
        
\end{proof}

\subsection{The main component of the Bilinear scheme}\label{r:secmaincomp}

Motivated by the smoothable component of the Hilbert scheme 
and the principal component of the Quot scheme, 
we define the component of the Bilinear scheme $\Bilin_{d_1,d_2,d_3}^{r_1, r_2}(\A^n)$
that can be viewed as parametrizing tuples of points. 
We will discuss the special case $\Bilin_{d,d,d}^{r_1, r_2} (\A^n)$. 
Recall that we have the closed embedding of functors 
$\iota \colon \Bilin_{d,d,d}^{r_1, r_2}(\A^n) \hookrightarrow \Q$. 
Consider the composition 
\[\begin{tikzcd}
	\Bilin_{d,d,d}^{r_1, r_2}(\A^n) & {\Quot_{d}^{r_1}(\A^n) \times \Quot_{d}^{r_2}(\A^n) \times \Quot_{d}^{r_1 r_2}(\A^n) } \\
	& {Quot_{d}^{r_1}(\A^n)}
	\arrow["\iota", hook, from=1-1, to=1-2]
	\arrow["{\pr_1}", from=1-2, to=2-2]
\end{tikzcd}\]
\noindent
which we will also denote by $\pr_1 \colon \Bilin_{d,d,d}^{r_1, r_2}(\A^n) \rightarrow \Quot_{d}^{r_1}(\A^n)$. 
For $\Bbbk$-points we will write 
\begin{equation*}
    \pr_1(\Bbbk) \colon \Bilin_{d,d,d}^{r_1, r_2}(\A^n)(\Bbbk) \rightarrow \Quot_{d}^{r_1}(\A^n)(\Bbbk).
\end{equation*}
Let $\U_d^{r_1} \subseteq \Quot_{d}^{r_1}(\A^n)$ denote the open subfunctor that gives 
the principal component of the Quot scheme. 
For the $\Bbbk$-points of the open subsheme $\U_d ^{r_1} (\Bbbk)$ 
consider the preimage $\pr_1(\Bbbk)^{-1}(\U_d ^{r_1}(\Bbbk))$. 
It coincides with the subset of $\Bilin_{d,d,d}^{r_1, r_2}(\A^n) (\Bbbk)$ defined as 
    \begin{equation*}
        \W(\Bbbk) = \{ S^{\oplus r_1} \xrightarrowdbl[]{p_1} M_1, \, S^{\oplus r_2} \xrightarrowdbl[]{p_2} M_2, \, M_1 \otimes M_2 \xrightarrowdbl[]{\pi} M_3 \mid M_i \text{ correspond to a tuple of points } \sfrac{S}{I} \} _{/ \sim}.
    \end{equation*}
\noindent
We will now describe the points of $\W(\Bbbk)$.

\begin{proposition}\label{r:prpalgiso}
    Suppose there is an isomorphism $M_1 \simeq \sfrac{S}{I}$. 
    Then also $M_2 \simeq M_3 \simeq \sfrac{S}{I}$. 
\end{proposition}

\begin{proof}
    Since $\dim_{\Bbbk} M_i=d$ for $i=1,2,3$ 
    and $M_1$ corresponds to a tuple of points, we can write 
    \begin{equation*}
            \Supp M_1 = \{ \mm_1, \dots, \mm_d \} \quad \text{and} \quad
            \Supp M_2 = \{ \nn_{1}, \dots, \nn_{s} \}. 
    \end{equation*}
    We have the decomposition 
    $M_1 \otimes M_2 \simeq \oplus _{i,j} (M_1)_{\mm_i} \otimes (M_2)_{\nn_j}$ 
    with $(M_1)_{\mm_i} \otimes (M_2)_{\nn_j} = 0$ for $\mm_i \neq \nn_j$. 
    Since $M_1 \simeq \sfrac{S}{I}$, we get the isomorphism 
    $M_1 \otimes M_2 \simeq \sfrac{M_2}{I M_2}$, 
    and thus $\dim _{\kk} M_1 \otimes M_2 \leq \dim_{\Bbbk} M_2 = d$.  
    The equality is necessary for the existence of the surjection 
    $M_1 \otimes M_2 \twoheadrightarrow M_3$, since $\dim_{\Bbbk} M_3 = d$. 
    Therefore $\Supp M_2 = \Supp M_1$. 
    The existence of $M_1 \otimes M_2 \twoheadrightarrow M_3$ implies 
    $\Supp M_3 \subseteq \Supp M_1 \otimes M_2$, 
    so $\Supp M_1 = \Supp M_2 = \Supp M_3$ 
    and since $\vert \Supp M_1 \vert = \dim_{\Bbbk} M_1$ 
    it must be $\dim_{\Bbbk}(M_1)_{\mm_i}=1$ for each $i=1, \ldots, d$. 
    It follows that $M_j \simeq \sfrac{S}{I}$ for $j=1,2,3$. 
\end{proof}

\begin{proposition}
    The map $M_1 \otimes M_2 \xrightarrowdbl[]{\pi} M_3$ is unique up to isomorphism, 
    hence the points of $\W(\Bbbk)$ are determined by the surjections 
    $S^{\oplus r_j} \xrightarrowdbl[]{p_j} M_j \simeq \sfrac{S}{I}$ for $j=1,2$. 
\end{proposition}

\begin{proof}
    We identify $\sfrac{S}{I} \otimes_S \sfrac{S}{I} \simeq \sfrac{S}{I}$, 
    so the map can be written as $\sfrac{S}{I} \xrightarrowdbl[]{\pi} \sfrac{S}{I}$, 
    which shows that it must be an isomorphism. 
    We identify maps with equal kernels, so since $\ker \pi =0$, 
    in this case there is only one equivalence class, as claimed. 
\end{proof}

\noindent
Since $\W(\Bbbk)$ arises as a continuous preimage 
of an open set $\U_d ^{r_1}(\Bbbk)$, it is also open. 
Therefore there is an open subscheme $W \subseteq Bilin$ such that $\W(\kk)=W(\kk)$, 
which we extend to the representable functor $\W = \underline{W}$, 
by Proposition \ref{r:prpkpointsext}. 
We define the main irreducible component of the Bilinear scheme as the closure of $W$. 

\begin{proposition}\label{r:prpbilindim}
    The dimension of $W$ is $nd + (r_1 -1)d + (r_2 - 1)d$. 
\end{proposition}

\noindent
In order to prove Proposition \ref{r:prpbilindim}, we will use \cite[Theorem 12.4.1]{vakil}. 
We do not need this result in full generality, 
it suffices to consider the special case stated in Corollary 12.4.2. 

\begin{proposition}[{\cite[Corollary 12.4.2]{vakil}}]\label{r:prpfibdim}
    Suppose $f \colon X \rightarrow Y$ is a finite type morphism 
    of irreducible $\Bbbk$-varieties. 
    Then there exists a non-empty open subset $V \subseteq Y$ 
    such that for all $p \in V$, the fiber over $p$ has dimension 
    $\dim X - \dim Y$, or is empty. 
\end{proposition}

\begin{proof}[Proof of Theorem \ref{r:prpbilindim}]
    We will apply Proposition \ref{r:prpfibdim} to our setting. 
    Let 
    \begin{equation*}
        [p_1] = \big[S^{\oplus r_1} \xrightarrowdbl[]{p_1} M_1 \simeq \sfrac{S}{I} \big] \in \U_d ^{r_1}(\Bbbk).
    \end{equation*} 
    Consider the fiber $\pr_1(\Bbbk)^{-1}([p_1])$ over $[p_1]$. 
    It is of the form 
    \begin{equation*}
        \pr_1(\Bbbk)^{-1}([p_1]) = \big[S^{\oplus r_1} \xrightarrowdbl[]{p_1} M_1, \, S^{\oplus r_2} \xrightarrowdbl[]{p_2} M_2, \, M_1 \otimes M_2 \xrightarrowdbl[]{\pi} M_3 \big],
    \end{equation*}
    where the map $S^{\oplus r_1} \xrightarrowdbl[]{p_1} M_1$ is fixed 
    and the map $M_1 \otimes M_2 \xrightarrowdbl[]{\pi} M_3$ is unique up to isomorphism. 
    It follows that the fiber is determined by the second map 
    $S^{\oplus r_2} \xrightarrowdbl[]{p_2} M_2$. 
    Restrict to an open affine neighborhood $\Spec A$ of $[p_1]$, 
    then Proposition \ref{r:prpue1iso}  
    gives an isomorphism of the fiber over $[p_1]$ with 
    $\Gamma(\Spec A, \, M_2)^{r_2-1} \simeq \A_{\Bbbk}^{r_2-1}(A)$. 
    By Proposition \ref{r:prpfibdim} we get 
    \begin{align*}
        \begin{split}
            \dim W & = \dim pr_1^{-1}([p_1]) + \dim \U_d ^{r_1} \\
            & = \dim \A_{\Bbbk}^{d(r_2-1)} + \dim \T + \dim \A_{\Bbbk}^{d(r_1-1)} \\
            & = d(r_2 -1) + nd + (r_1-1)d.
        \end{split}
    \end{align*}

\end{proof}

\subsection{The totally degenerate locus of the Bilinear scheme}\label{r:secuglycomp}

Assume $r, r_1,r_2 \geq d$. 
Fix the maximal ideal $\mm = (x_1, \dots, x_n) \ \lhd \ S = \kk [ x_1, \dots, x_n ]$. 
We now define the locus that corresponds to 
surjections of $S$-modules of the form $(\sfrac{S}{\mm})^{\oplus d}$. 
This is a generalization of the totally degenerate locus of the Quot scheme. 
Consider the subset of $ \Bilin_{d,d,d}^{r_1, r_2}(\A^n) (\kk)$ that consists of points of the form 
\begin{equation}\label{r:eqzmbilin}
    [q_1,q_2, \pi] = \big[S^{\oplus r_1} \xrightarrowdbl[]{q_1} M_1 \simeq (\sfrac{S}{\mm})^{\oplus d}, \, S^{\oplus r_2} \xrightarrowdbl[]{q_2}  M_2 \simeq (\sfrac{S}{\mm})^{\oplus d}, \, M_1 \otimes M_2 \xrightarrowdbl[]{\pi} M_3 \big].
\end{equation}
\noindent
We have assumed $r_i \geq d$, so the surjections $q_i$ exist. From the isomorphism
\begin{equation}\label{r:eqdsquare}
    M_1 \otimes M_2 \simeq (\sfrac{S}{\mm})^{\oplus d} \otimes (\sfrac{S}{\mm})^{\oplus d} \simeq (\sfrac{S}{\mm})^{\oplus d^2}
\end{equation}
\noindent 
the module $M_3$ must also be isomorphic to $(\sfrac{S}{\mm})^{\oplus d}$. 
The resulting subset is  
\[
\Zz_{\mm}(\Bbbk) = 
\left\{\begin{minipage}[c]{\widthof{\text{$S^{\oplus r_i} \xrightarrowdbl[]{q_i} M_i \simeq (\sfrac{S}{\mm})^{\oplus d}$ for $i=1,2$}}}
\centering
\text{$S^{\oplus r_i} \xrightarrowdbl[]{q_i} M_i \simeq (\sfrac{S}{\mm})^{\oplus d}$ for $i=1,2$}\\
\text{$M_1 \otimes_S M_2 \xrightarrowdbl[]{\pi} M_3 \simeq (\sfrac{S}{\mm})^{\oplus d}$}
\end{minipage}\right\}_{/ \sim}
\]
\noindent
Our goal is to define a closed subfunctor $\Zz_{\mm}$ of $\Bilin_{d,d,d}^{r_1, r_2}(\A^n)$ 
whose $\kk$-points coincide with this subset. 
In order to achieve this, we define a closed subfunctor 
$\Zz_{\mm}^{\Q} \subseteq \Quot_d^{r_1}(\A^n) \times \Quot_d^{r_2}(\A^n)$ 
and show that $\Zz_{\mm}(\kk)$ coincides with the $\kk$-points 
of the preimage of $\Zz_{\mm}^{\Q}$ along the map defined as the composition 
\begin{equation}\label{r:eqrho}
    \rho \colon \Bilin_{d,d,d}^{r_1, r_2}(\A^n) \hookrightarrow 
    \Q \twoheadrightarrow  
    \Quot_d^{r_1}(\A^n) \times \Quot_d^{r_2}(\A^n), 
\end{equation}
where $\Q = \Quot_d^{r_1}(\A^n) \times \Quot_d^{r_2}(\A^n) \times \Quot_d^{r_1 r_2}(\A^n)$. 
To define the subfunctor $\Zz_{\mm}^{\Q}$, notice that surjections of the form 
$S^{\oplus r} \xrightarrowdbl[]{q} M \simeq (\sfrac{S}{\mm})^{\oplus d}$ 
define a subset of $\Quot_d^r(\A^n)(\kk)$ that consists of $\kk$-points 
of $\Zz_{\mm} \subseteq \Quot_d^{r_1}(\A^n)$. 
Taking the product
\begin{equation*}
    \Zz_{\mm}^{\Q}(\kk)= \Zz_{\mm}(\kk) \times \Zz_{\mm}(\kk) \subseteq (\Quot_d^{r_1}(\A^n) \times \Quot_d^{r_2}(\A^n))(\kk)
\end{equation*}  
yields the desired functor $\Zz_{\mm}^{\Q}= \Zz_{\mm} \times \Zz_{\mm}$. 
The points of $\Zz_{\mm}^{\Q}(\kk)$ are of the form 
\begin{align*}
    [q_1,q_2] = \big[S^{\oplus r_1} \xrightarrowdbl[]{q_1} M_1 \simeq (\sfrac{S}{\mm})^{\oplus d}, \, S^{\oplus r_2} \xrightarrowdbl[]{q_2}  M_2 \simeq (\sfrac{S}{\mm})^{\oplus d} \big].
\end{align*}
By Proposition \ref{r:prpgrzmiso} we have the isomorphism 
$\Zz_{\mm}^{\Q} \simeq \Gr(d,r_1) \times \Gr(d, r_2)$. 
Consider the preimage $\Zz_{\mm}^{\B}(\kk) = \rho(\kk)^{-1}(\Zz_{\mm}^{\Q}(\kk))$, 
where $\rho(\kk)$ denotes the map induced on $\kk$-points by $\rho$ defined in (\ref{r:eqrho}). 
By the isomorphism (\ref{r:eqdsquare}) the points of $\Zz_{\mm}^{\B}(\kk)$ are of the form 
\[
\Zz_{\mm}^{\B}(\Bbbk) = 
\left\{\begin{minipage}[c]{\widthof{\text{$S^{\oplus r_i} \xrightarrowdbl[]{q_i} M_i \simeq (\sfrac{S}{\mm})^{\oplus d}$ for $i=1,2$}}}
\centering
\text{$S^{\oplus r_i} \xrightarrowdbl[]{q_i} M_i \simeq (\sfrac{S}{\mm})^{\oplus d}$ for $i=1,2$}\\
\text{$M_1 \otimes_S M_2 \xrightarrowdbl[]{\pi} M_3 \simeq (\sfrac{S}{\mm})^{\oplus d}$}
\end{minipage}\right\}_{/ \sim}
\]
\noindent
which coincides with the definition of the subset 
$\Zz_{\mm}(\kk) \subseteq \Bilin_{d,d,d}^{r_1, r_2}(\A^n)(\kk)$ we have given in (\ref{r:eqzmbilin}). 
It follows that this subset is closed. 
Next we define the corresponding closed subfunctor 
$\Zz_{\mm} = \Zz_{\mm}^{\B}$ of $\Bilin_{d,d,d}^{r_1,r_2}(\A^n)$. 
We will use the results from Section \ref{r:secquot} to show that 
$\Zz_{\mm} \hookrightarrow \Bilin_{d,d,d}^{r_1,r_2}(\A^n)$ 
is of dimension $(r_1-d)d + (r_2 -d)d + (d^2 - d)d$.
Let 
\begin{equation*}
    \bar{\rho} = \rho_{\vert _{\Zz_{\mm}^{\B}}} \colon \Zz_{\mm}^{B} \rightarrow \Zz_{\mm}^{\Q}, \quad
    \text{where } \ \Zz_{\mm}^{\Q} = \Zz_{\mm} \times \Zz_{\mm} \subseteq  \Quot_d^{r_1}(\A^n) \times \Quot_d^{r_2}(\A^n),
\end{equation*} 
as before. 
Take a point $[q_1,q_2] \in \Zz_{\mm}^{\Q}(\kk)$, where 
$q_i \colon S^{\oplus r_i} \twoheadrightarrow M_i \simeq (\sfrac{S}{\mm})^{\oplus d}$ for $i=1,2$. 
Consider the fiber $\bar{\rho}(\kk)^{-1}([q_1,q_2]) \in \Zz_{\mm}^{\B}(\kk)$. 
Since we have the isomorphism $M_1 \otimes M_2 \simeq (\sfrac{S}{\mm})^{\oplus d^2}$ 
and the assumption $\dim_{\kk} M_3 = d$, there must also be an isomorphism $M_3 \simeq (\sfrac{S}{\mm})^{\oplus d}$. 
Hence the fiber over $[q_1,q_2]$ consists of $S$-linear surjections of the form 
$\pi \colon (\sfrac{S}{\mm})^{\oplus d^2} \twoheadrightarrow (\sfrac{S}{\mm})^{\oplus d}$. 

\begin{proposition}
    Let $[q_1,q_2] \in \Zz_{\mm}^{\Q}(\kk)$ be a point corresponding to surjections of the form 
    $q_i \colon S^{\oplus r_i} \twoheadrightarrow M_i \simeq (\sfrac{S}{\mm})^{\oplus d}$ for $i=1,2$. 
    Then the fiber $\bar{\rho}(\kk)^{-1}([q_1,q_2])$ 
    is isomorphic to $\Gr(d, d^2)(\kk)$. 
\end{proposition}

\begin{proof}
    The isomorphism $\alpha_{\mm} \colon \kk \simeq \sfrac{S}{\mm}$ gives 
    a map $\phi \colon \Gr(d,d^2)(\kk) \rightarrow \bar{\rho}(\kk)^{-1}([q_1,q_2])$ 
    defined as 
    \begin{equation*}
        \big[\kk^{\oplus d^2} \xrightarrowdbl[]{\pi} V \simeq \kk^{\oplus d} \big] \mapsto 
        \big[(\sfrac{S}{\mm})^{\oplus d^2} \simeq \kk^{\oplus d^2} \xrightarrowdbl[]{\pi} V \simeq \kk^{\oplus d} \simeq (\sfrac{S}{\mm})^{\oplus d} \big].
    \end{equation*}
    This is clearly a bijection, so it gives the desired isomorphism. 
\end{proof}

\begin{proposition}\label{r:prpzmdim}
    The dimension of the closed subscheme $\Zz_{\mm}^{\B} \subseteq \Bilin_{d,d,d}^{r_1, r_2}(\A^n)$ 
    is $(r_1-d)d + (r_2 -d)d + (d^2 - d)d$.
\end{proposition}

\begin{proof}
    By Proposition \ref{r:prpfibdim},  
    \begin{align*}
        \begin{split}
            \dim \Zz_{\mm}^{\B} & = \dim \Zz_{\mm}^{\Q} + \dim \bar{\rho}(\kk)^{-1}([q_1,q_2]) \\
            & = \dim (\Gr(d,r_1) \times \Gr(d, r_2)) + \dim \Gr(d,d^2) \\
            & = (r_1-d)d + (r_2 -d)d + (d^2 - d)d.
        \end{split}
    \end{align*}
\end{proof}

\subsection{Irreducibility}\label{r:secbilinisredu}

We will use Propositions \ref{r:prpbilindim} and \ref{r:prpzmdim} to find examples of parameters $n, r_i, d$ 
such that the Bilinear scheme $\Bilin_{d,d,d}^{r_1,r_2}(\A^n) $ is not irreducible. 
Let $Z_1$ denote the closure of the open subscheme $W \subseteq Bilin$ 
corresponding to surjections onto cyclic modules, as defined in Section \ref{r:secmaincomp}. 
Let $Z_2$ denote the closed subscheme representing the subfunctor $\Zz_{\mm}^{\B}$, as defined in Section \ref{r:secuglycomp}. 
By Lemma \ref{r:lemirreddim}, if $n, r_i, d$ are such that 
$\dim Z_1 < \dim Z_2$, then $Z_2$ must be contained in a different component than $Z_1$ and 
$\Bilin_{d,d,d}^{r_1,r_2}(\A^n)$ is reducible. 
Examples of parameters $r_i, d$ enjoying this property arise already for $n=1,2$. 

\begin{theorem}\label{r:thmredu}
    Let $r_i \geq d \geq 3$. If $n < d^2 - 3d +2$, then 
    $\Bilin_{d,d,d}^{r_1,r_2}(\A^n)$ is reducible. 
\end{theorem}

\begin{proof}
    Given $r_i, d, n$ as in the statement, we have 
    \begin{equation*}
        \dim Z_1 = (n + (r_1 -1) + (r_2-1))d < (r_1+r_2+d^2 -3d)d = \dim Z_2.
    \end{equation*}
\end{proof}

It remains to consider the case of $\Bilin_{d,d,d}^{r_1,r_2}(\A^n)$ with $n \geq d^2 - 3d +2$. 
We will use the results on secant varieties from Section \ref{r:sectensors} to improve Theorem \ref{r:thmredu}. 
For $[p_1,p_2,\pi] \in \Bilin_{d,d,d}^{r_1,r_2}(\A^n)(\kk)$, 
we have $[p_1,p_2,\pi] \in Z_1$ if and only if the corresponding tensor 
$\mu_{\pi} \in M_1^{\vee} \otimes M_2^{\vee} \otimes M_3$ is of border rank $d$. 
Therefore, claiming that $\Bilin_{d,d,d}^{r_1,r_2}(\A^n)$ is irreducible is the same as 
saying that all tensors $\kk ^d \otimes \kk^d \otimes \kk^d$ are of border rank $d$. 
This yields the following necessary condition for irreducibility of $\Bilin_{d,d,d}^{r_1,r_2}(\A^n)$:  
the variety $\sigma_d (\PP(\kk^d) \times \PP(\kk^d) \times \PP(\kk^d))$ must coincide with 
the ambient space $\PP(\kk^d \otimes \kk^d \otimes \kk^d)$. 
We prove that this condition is not satisfied whenever $d \geq 3$. 

\begin{theorem}
    Let $r_i \geq d \geq 3$. Then $\Bilin_{d,d,d}^{r_1,r_2}(\A^n) $ is reducible. 
\end{theorem}
\begin{proof}
    It suffices to show that 
    \begin{equation*}
        \sigma_d(\PP (\kk^d)^3) \neq \mathbb{P}^{d^3 - 1} = \PP (\kk^d \otimes \kk^d \otimes \kk^d).
    \end{equation*}
    By Proposition \ref{r:prpsecantdim} we have 
    \begin{equation*}
        \dim \sigma_d(\PP (\kk^d)^3) \leq \min \, \{ d^3, \, d(3d +1)-1 \}.
    \end{equation*}
    For $d \leq 3$ we have $d^3 \leq d(3d+1)-1$, so 
    $\dim \sigma_d(\PP (\kk^d)^3) < \dim \mathbb{P}^{d^3 - 1}$. 
    It means that there exist tensors $\mu \in M_1^{\vee} \otimes M_2^{\vee} \otimes M_3$ 
    of border rank greater than $d$. 
    It remains to show that we can choose those tensors to be $M_3$-concise. 
    By Lemma \ref{r:lemconopen}, concise tensors form an open subset of $\mathbb{P}^{d^3 - 1}$, 
    and since this space is irreducible, the subspace of concise tensors is dense. 
    Hence there must be a concise tensor 
    $\mu \in M_1^{\vee} \otimes M_2^{\vee} \otimes M_3$ 
    not contained in $\sigma_d(\PP (\kk^d)^3)$. 
    This tensor defines a map $\pi_{\mu} \colon M_1 \otimes M_2 \twoheadrightarrow M_3$ 
    that gives a point $[p_1, p_2, \pi_{\mu}] \in \Bilin_{d,d,d}^{r_1,r_2}(\A^n) \setminus \overline{\mathcal{W}}$.
\end{proof}

Now we will show that $\Bilin_{2,2,2}^{r_1,r_2}(\A^n) $ is irreducible for all $n$. 
We break it down into two steps. 
\begin{enumerate}
    \item Show that every $S$-linear quotient $p \colon S^{\oplus r} \twoheadrightarrow M$ 
    with $\dim_{\kk} M =2$, arises as a limit of quotients corresponding to tuples of points. 
    This is equivalent to proving that $\Quot_2 ^r(\A^n)$ is irreducible, Proposition \ref{r:prpquotirred}.

    \item Given two $S$-linear quotients $p_i \colon S^{\oplus r_i} \twoheadrightarrow M_i$ 
    together with $\pi \colon M_1 \otimes_S M_2 \twoheadrightarrow M_3$,
    such that $\dim_{\kk} M_i = \dim_{\kk} M_3 = 2$,  
    show that the corresponding tensor $\mu_{\pi} \in M_1^{\vee} \otimes M_2^{\vee} \otimes M_3$ is of border rank 2.
\end{enumerate}
\noindent 
The second step follows from the following result. 

\begin{proposition}\label{r:prpbdrk2}
    Let $X = \PP(\kk^2) \times \PP(\kk^2) \times \PP(\kk^2) \hookrightarrow \PP(\kk^2 \otimes \kk^2 \otimes \kk^2 )$ be 
    the Segre embedding. Then $\sigma_2(X) = \PP(\kk^2 \otimes \kk^2 \otimes \kk^2 )$.
\end{proposition}

\begin{theorem}
    For every $n \geq 1$ and $r_1,r_2 \geq 2$ the scheme representing $\Bilin_{2,2,2}^{r_1,r_2}(\A^n)$ is irreducible.
\end{theorem}

\begin{proof}
    By Proposition \ref{r:prpbdrk2}, every $\pi \colon M_1 \otimes M_2 \twoheadrightarrow M_3$ 
    corresponds to a tensor of border rank 2. To define a point of $\Bilin_{2,2,2}^{r_1,r_2}(\A^n)$ 
    we also have to specify $p_i \colon S^{\oplus r_i} \twoheadrightarrow M_i$ and show that 
    they agree with taking the limit. The existence of those maps follows from Proposition \ref{r:prpquotirred}.
\end{proof}

\subsection{Two points}

We describe the $\kk$-points of $\Bilin_{2,2,2}^{2,2}(\A^1) $ in detail. 
Let $\A^1 = \Spec S$, where $S = \kk [x]$, and fix the maximal ideal $\mm = (x) \lhd S$. 
A point $[p_1,p_2,\pi] \in \Bilin_{2,2,2}^{2,2}(\A^1)$ is given by 
\begin{equation*}
    p_i \colon S^{\oplus 2} \twoheadrightarrow M_i \, (i=1,2), \quad \pi \colon M_1 \otimes M_2 \twoheadrightarrow M_3, 
\end{equation*}
such that $\dim_{\kk} M_j = 2$ for $j=1,2,3$. 
To give an $S$-mod $M$ with $\dim_{\kk}=2$ is the same as giving a $\kk$-vector space of dimension 2 
together with an endomorphism corresponding to multiplication by $x$. 
We find that $M$ satisfies one of the following:
\begin{enumerate}
    \item $M \simeq \sfrac{S}{((x-c_1)(x-c_2))}$ for $c_1,c_2 \in \kk$ such that $c_1 \neq c_2$. 
    \item $M \simeq \sfrac{S}{(x-c)^2}$ for $c \in \kk$.
    \item $M \simeq (\sfrac{S}{(x-c)})^{\oplus 2}$ for $c \in \kk$.
\end{enumerate}
\noindent
Without loss of generality, we can consider cases $\sfrac{S}{(x(x-1))}$, $\sfrac{S}{(x)^2}$ and $(\sfrac{S}{(x)})^{\oplus 2}$. 
Let $e_1,e_2$ be the basis of $S^{\oplus 2}$ with each $e_i$ corresponding to $1 \in S$. 

By Proposition \ref{r:prpalgiso}, if $p_1 \colon S^{\oplus 2} \twoheadrightarrow M_1 \simeq \sfrac{S}{I}$ 
and $p_2 \colon S^{\oplus 2} \twoheadrightarrow M_2 \simeq \sfrac{S}{J}$ for 
some ideals $I,J$ in $S$, then it must be $M_1 \simeq M_2 \simeq M_3 \simeq \sfrac{S}{I}$. 
For $p_i \colon S^{\oplus 2} \twoheadrightarrow M_i \simeq \sfrac{S}{(x(x-1))}$, 
the third map $\pi_1 \colon M_1 \otimes M_2 \twoheadrightarrow M_3$ is a uniquely defined isomorphism 
corresponding to a concise tensor of minimal border rank, that is, of border rank 2. 
By Proposition \ref{r:prpunitorb} we can take the unit tensor. 
Let $\alpha_1, \alpha_2 $ be the basis of $\sfrac{S}{(x(x-1))} \simeq \sfrac{S}{(x)} \times \sfrac{S}{(x-1)}$, 
with $\alpha_1$ and $\alpha_2$ corresponding to $1 \in \sfrac{S}{(x)}$ and $1 \in \sfrac{S}{(x-1)}$, respectively. 
Then we can assume that $\pi_1$ is the multiplication in $\sfrac{S}{(x(x-1))}$. 
In basis $\alpha_1, \alpha_2 $ the corresponding tensor is 
\begin{equation*}
    \mu_1 = \alpha_1^{*} \otimes \alpha_1^{*} \otimes \alpha_1 + \alpha_2^{*} \otimes \alpha_2^{*} \otimes \alpha_2.
\end{equation*}

For $p_i \colon S^{\oplus 2} \twoheadrightarrow M_i \simeq \sfrac{S}{\mm^2}$ we have 
$M_1 \otimes_S M_2 \simeq \sfrac{S}{\mm^2} \otimes_S \sfrac{S}{\mm^2} \simeq \sfrac{S}{\mm^2} \xrightarrowdbl[]{\pi_2} \sfrac{S}{\mm^2}$, 
so the third map $\pi_2$ must be an isomorphism, which we can assume to be 
the multiplication in $\sfrac{S}{\mm^2}$. 
Let $\beta_1, \beta_2$ be the basis of $\sfrac{S}{\mm^2}$ 
with $\beta_1$ corresponding to $1 \in \sfrac{S}{\mm^2}$ and 
$\beta_2$ corresponding to $x \in \sfrac{S}{\mm^2}$. 
Then the tensor corresponding to $\pi_2$ is 
\begin{equation*}
    \mu_2 = \beta_1 ^{*} \otimes \beta_1 ^{*} \otimes \beta_1 
    + \beta_1 ^{*} \otimes \beta_2 ^{*} \otimes \beta_2
    + \beta_2 ^{*} \otimes \beta_1 ^{*} \otimes \beta_2, 
\end{equation*} 
concise of rank 3. 
By Proposition \ref{r:prpbdrk2}, $\mu_2$ is of border rank 2, 
so there is a tensor $\mu_2(t)$ of rank 2 such that the corresponding map $\pi_2(t)$ 
is isomorphic to the multiplication in $\sfrac{S}{((x-c_1(t))(x-c_2(t)))}$. 
Let $M_t \simeq \sfrac{S}{I_t} = \sfrac{S}{(x(x-t))}$ with basis 
$\theta_1, \theta_2$ corresponding to $1, x \in \sfrac{S}{I_t}$, respectively. 
Consider the multiplication 
\begin{equation*}
    \pi_2(t) \colon M_t \otimes M_t \rightarrow M_t.
\end{equation*}
The tensor $\mu_2(t)$ in basis $\theta_1, \theta_2$ is 
\begin{equation*}
    \mu_2(t) = \theta_1^{*} \otimes \theta_1^{*} \otimes \theta_1 
    + \theta_1^{*} \otimes \theta_2^{*} \otimes \theta_2
    + \theta_2^{*} \otimes \theta_1^{*} \otimes \theta_2
    + \theta_2^{*} \otimes \theta_2^{*} \otimes t \theta_2.
\end{equation*}
Letting $t \rightarrow 0$, we get $\mu_2 = \lim\limits_{t \rightarrow 0} \, \mu_2(t)$,
and since $\mu_2(t)$ is isomorphic to the unit tensor, 
this shows that $\mu_2$ is of border rank 2.
By Proposition \ref{r:prpquotirred}, for any $p_i \colon S^{\oplus 2} \twoheadrightarrow M_i \simeq \sfrac{S}{\mm^2}$,
there exist $p_i(t) \colon S^{\oplus 2} \twoheadrightarrow M_t$ such that 
\begin{equation*}
    [p_1, p_2, \pi_2] = \lim\limits_{t \rightarrow 0} \, [p_1(t), p_2(t), \pi_2(t)].
\end{equation*}

Now we consider $M_1 \simeq \sfrac{S}{I}$ without assuming that $M_2$ is cyclic. 
The only case we need to discuss is 
\begin{equation*}
    p_1 \colon S^{\oplus 2} \twoheadrightarrow M_1 \simeq \sfrac{S}{\mm^2}. 
\end{equation*}
Since $\Supp M_1 \otimes M_2 = \Supp M_1 \cap \Supp M_2$, it must be 
$\Supp M_2 = \{ \mm \}$. Then we have either 
$M_2 \simeq \sfrac{S}{\mm^2}$ or $M_2 \simeq (\sfrac{S}{\mm})^{\oplus 2}$. 
We have already covered the former case, so we assume the latter. 
From the isomorphisms 
\begin{equation*}
    M_1 \otimes M_2 \simeq \sfrac{S}{\mm^2} \otimes (\sfrac{S}{\mm})^{\oplus 2} \simeq (\sfrac{S}{\mm})^{\oplus 2} 
\end{equation*}
we deduce that $M_3 \simeq (\sfrac{S}{\mm})^{\oplus 2}$ and that the third map 
$\pi_3 \colon M_1 \otimes M_2 \twoheadrightarrow M_3$ is an isomorphism. 
Let $\gamma_1, \gamma_2$ be the basis of $(\sfrac{S}{\mm})^{\oplus 2}$ 
with each $\gamma_i$ corresponding to $1 \in \sfrac{S}{\mm}$, 
and $\beta_1, \beta_2$ the basis of $\sfrac{S}{\mm^2}$ as before. 
Then the tensor corresponding to 
$\pi_3 \colon M_1 \otimes M_2 \twoheadrightarrow M_3$ is isomorphic to 
\begin{equation*}
    \mu_3 = \beta_1^{*} \otimes \gamma_1^{*} \otimes \gamma_1 +
    \beta_1^{*} \otimes \gamma_2^{*} \otimes \gamma_2,
\end{equation*}
which is the multiplication tensor in the $S$-module $(\sfrac{S}{\mm})^{\oplus 2}$. 
Note that this tensor is not $M_1^{\vee}$-concise. Likewise, given 
\begin{equation*} 
    p_1 \colon S^{\oplus 2} \twoheadrightarrow M_1 \simeq (\sfrac{S}{\mm})^{\oplus 2} \quad \text{and} \quad 
    p_2 \colon S^{\oplus 2} \twoheadrightarrow M_2 \simeq \sfrac{S}{\mm^2},
\end{equation*}
we see that the third map 
$\pi_4 \colon M_1 \otimes M_2 \twoheadrightarrow M_3 \simeq (\sfrac{S}{\mm})^{\oplus 2}$ 
corresponds to 
\begin{equation*}
    \mu_4 = \gamma_1^{*} \otimes \beta_1^{*} \otimes \gamma_1 +
    \gamma_2^{*} \otimes \beta_1^{*} \otimes \gamma_2,
\end{equation*}
which fails to be concise on the second coordinate. 
To see that $\mu_3$ arises as a limit of concise rank 2 tensors, 
consider $M_{t,1} \simeq \sfrac{S}{(x(x-t))}$ with basis $\theta_1, \theta_2$ corresponding to 
$1, x \in \sfrac{S}{(x(x-t))}$, and $M_{t,2} \simeq \sfrac{S}{(x)} \times \sfrac{S}{(x-t)}$ 
with basis $\alpha_1, \alpha_2$ corresponding to $1 \in \sfrac{S}{(x)}$, $\sfrac{S}{(x-t)}$, respectively. 
Then 
\begin{align*}
    \begin{split}
        \mu_3(t) & = \theta_1^{*} \otimes \alpha_1^{*} \otimes \alpha_1 +
        \theta_1^{*} \otimes \alpha_2^{*} \otimes \alpha_2 +
        \theta_2^{*} \otimes \alpha_2^{*} \otimes t \alpha_2 \\
        & = \theta_1^{*} \otimes \alpha_1^{*} \otimes \alpha_1 +
        (\theta_1^{*} + t \theta_2^{*}) \otimes \alpha_2^{*} \otimes \alpha_2
    \end{split}
\end{align*} 
is the desired tensor. Likewise, for 
$\pi_4 \colon (\sfrac{S}{\mm})^{\oplus 2} \otimes \sfrac{S}{\mm^2} \twoheadrightarrow (\sfrac{S}{\mm})^{\oplus 2}$, 
we obtain $\mu_4$ as the limit of 
\begin{align*}
    \begin{split}
        \mu_4(t) & = \alpha_1^{*} \otimes \theta_1^{*} \otimes \alpha_1 +
        \alpha_2^{*} \otimes \theta_1^{*} \otimes \alpha_2 +
        \alpha_2^{*} \otimes \theta_2^{*} \otimes t \alpha_2 \\
        & = \alpha_1^{*} \otimes \theta_1^{*} \otimes \alpha_1 +
        \alpha_2^{*} \otimes (\theta_1^{*} + t \theta_2^{*}) \otimes \alpha_2.
    \end{split}
\end{align*} 
In both cases, the map
\begin{equation*}
    M_{t,i} \rightarrow M_{t,i}, \quad \theta_1 \mapsto \theta_1, \, \theta_2 \mapsto \theta_1 + t \theta_2
\end{equation*}
is an isomorphism for $t \neq 0$, so it is an isomorphism for $t=0$. 

It remains to consider a point $[p_1, p_2, \pi_5] \in \Zz_{\mm}(\kk)$, given by 
\begin{equation*}
    p_i \colon S^{\oplus 2} \twoheadrightarrow M_i \simeq (\sfrac{S}{\mm})^{\oplus 2} \ (i=1,2), \quad 
    \pi \colon M_1 \otimes M_2 \simeq (\sfrac{S}{\mm})^{\oplus 4} \twoheadrightarrow M_3 \simeq (\sfrac{S}{\mm})^{\oplus 2}.
\end{equation*}
This time the third map is not an isomorphism, 
but a surjection defined by a full rank $4 \times 2$ matrix. 
For any $p_i \colon S^{\oplus 2} \twoheadrightarrow M_i$ 
there are maps $p_i(t) \colon S^{\oplus 2} \twoheadrightarrow M_{t,i}$ 
such that each $M_{t,i}$ corresponds to a tuple of points 
and $p_i = \lim\limits_{t \rightarrow 0} \, p_i(t)$. 
No new types of tensors can arise from $\pi_5$, we have already covered every possibility. 
If the tensor corresponding to $\pi_5$ is concise, then it is isomorphic to 
the multiplication tensor of some algebra $\sfrac{S}{I}$, 
which can be either of the form $\sfrac{S}{(x-c)^2}$ 
or of the form $\sfrac{S}{(x-c_1)(x-c_2)}$, for $c,c_1,c_2 \in \kk$ such that $c_1 \neq c_2$. 
Otherwise the tensor fails to be concise on the first or on the second coordinate. 
If the non-conciseness occurs on the first coordinate, then the tensor is isomorphic to
\begin{equation*}
    \gamma_1^{*} \otimes \gamma_1^{*} \otimes \gamma_1 + 
    \gamma_1^{*} \otimes \gamma_2^{*} \otimes \gamma_2 = 
    \lim\limits_{t \rightarrow 0} \, 
    (\gamma_1^{*} \otimes \gamma_1^{*} \otimes \gamma_1 + 
    (\gamma_1^{*} + t \gamma_2^{*}) \otimes \gamma_2^{*} \otimes \gamma_2). 
\end{equation*}
If it fails to be concise on the second coordinate, then it is isomorphic to 
\begin{equation*}
    \gamma_1^{*} \otimes \gamma_1^{*} \otimes \gamma_1 + 
    \gamma_2^{*} \otimes \gamma_1^{*} \otimes \gamma_2 = 
    \lim\limits_{t \rightarrow 0} \, 
    (\gamma_1^{*} \otimes \gamma_1^{*} \otimes \gamma_1 + 
    \gamma_2^{*} \otimes (\gamma_1^{*} + t \gamma_2^{*}) \otimes \gamma_2). 
\end{equation*}
Those last two cases come from the multiplication in $(\sfrac{S}{\mm})^{\oplus 2} \in S$-mod.

\bibliographystyle{alpha}
\bibliography{references.bib}

\end{document}